\documentclass[11pt]{article}
\pdfoutput=1
\usepackage{float}
\usepackage[english]{babel}
\usepackage[margin=1in]{geometry}
\usepackage{background}
\usepackage{amssymb}
\usepackage{amsthm}
\usepackage{dsfont}
\usepackage{tikz}
\usepackage{adjustbox}
\usetikzlibrary{decorations.pathreplacing,decorations.markings}
\tikzset{
  on each segment/.style={
    decorate,
    decoration={
      show path construction,
      moveto code={},
      lineto code={
        \path [#1]
        (\tikzinputsegmentfirst) -- (\tikzinputsegmentlast);
      },
      curveto code={
        \path [#1] (\tikzinputsegmentfirst)
        .. controls
        (\tikzinputsegmentsupporta) and (\tikzinputsegmentsupportb)
        ..
        (\tikzinputsegmentlast);
      },
      closepath code={
        \path [#1]
        (\tikzinputsegmentfirst) -- (\tikzinputsegmentlast);
      },
    },
  },
  mid arrow/.style={postaction={decorate,decoration={
        markings,
        mark=at position .5 with {\arrow[#1]{stealth}}
      }}},
}
\usepackage{amsmath, amsfonts}
\usepackage{scrextend}
  \usepackage{paralist}
  \usepackage{graphics} 
  \usepackage{epsfig} 
\usepackage{graphicx}
\usepackage{epstopdf}
\usepackage[font={small,it}]{caption}
\usepackage{subcaption}
 \usepackage[colorlinks=true]{hyperref}
\hypersetup{urlcolor=blue, citecolor=red}

\SetBgContents{}

\DeclareMathOperator{\id}{id}
\DeclareMathOperator{\Id}{Id}

\DeclareMathOperator{\rank}{rk}
\DeclareMathOperator{\diam}{diam}

\DeclareMathOperator{\erfc}{erfc}

\DeclareMathOperator{\ess}{ess}

\DeclareMathOperator*{\esssup}{ess\,sup}
\DeclareMathOperator*{\essinf}{ess\,inf}

\newtheorem{theorem}{Theorem}[section]
\newtheorem*{theorem-nonA}{Theorem A}
\newtheorem*{theorem-nonB}{Theorem B}
\newtheorem*{theorem-nonC}{Theorem C}
\newtheorem*{theorem-nonD}{Conjecture D}
\newtheorem*{theorem-nonE}{Theorem E}
\newtheorem*{theorem-nonF}{Theorem F}

\newtheorem{lemma}[theorem]{Lemma}
\newtheorem{proposition}[theorem]{Proposition}

\theoremstyle{definition}
\newtheorem{definition}[theorem]{Definition}
\newtheorem{remark}[theorem]{Remark}

\newcommand{\rmd}{\mathrm{d}}
\newcommand{\rmD}{\mathrm{D}}
\renewcommand{\epsilon}{\varepsilon}

\newcommand{\fa}          {\quad \text{for all } \,}
\newcommand{\faa}          {\quad \text{for almost all } \,}
\newcommand{\R} {\mathbb R}
\numberwithin{equation}{section}

\begin{document}

\title{\vspace*{-10mm}
Hopf bifurcation with additive noise}
\author{Thai Son Doan\thanks{Institute of Mathematics, Vietnam Academy of Science and Technology, 18 Hoang Quoc Viet, Ha Noi, Vietnam} \and Maximilian Engel\thanks{Department of Mathematics, Imperial College London, 180 Queen’s Gate, London SW7 2AZ, United Kingdom
} \and Jeroen S.W.~Lamb\footnotemark[2] \and Martin~Rasmussen\footnotemark[2]}
\date{\today}
\maketitle

\begin{abstract}
We consider the dynamics of a two-dimensional ordinary differential equation exhibiting a Hopf bifurcation subject to additive white noise and
identify three dynamical phases: (I) a random attractor with uniform synchronisation of trajectories, (II) a random attractor with non-uniform synchronisation of trajectories and (III) a random attractor without synchronisation of trajectories. The random attractors in phases (I) and (II)  are random equilibrium points with negative Lyapunov exponents while in phase (III)  there is a so-called random strange attractor with positive Lyapunov exponent.

We analyse the occurrence of the different dynamical phases as a function of the linear stability of the origin (deterministic Hopf bifurcation parameter) and shear (ampitude-phase coupling parameter). We show that small shear implies synchronisation and obtain that synchronisation cannot be uniform in the absence of linear stability at the origin or in the presence of sufficiently strong shear. We provide numerical results in support of a conjecture that irrespective of the linear stability of the origin, there is a critical strength of the shear at  which the system dynamics loses synchronisation and enters phase (III).


\bigskip

\noindent $\textit{Key words}$. Dichotomy spectrum, Hopf bifurcation, Lyapunov exponent, random attractor, random dynamical system, stochastic bifurcation

\noindent $\textit{Mathematics Subject Classification (2010)}$.
37C75, 37D45, 37G35, 37H10, 37H15.

\end{abstract}
\section{Introduction}
We consider the two-dimensional stochastic differential equation
\begin{equation}\label{NormalForm}
\begin{array}{rl}
&\rmd x = (\alpha x - \beta y - (ax-by)(x^2 + y^2))\,\rmd t + \sigma \,\rmd W_t^1\,,\\
&\rmd y =  (\alpha y + \beta x - (bx+ay)(x^2 + y^2))\, \rmd t +  \sigma \,\rmd W_t^2\,,
\end{array}
\end{equation}
where $\sigma \geq 0$ represents the strength of the noise, $\alpha\in\mathbb{R}$ is a parameter equal to the real part of eigenvalues of the linearization of the vector field at $(0,0)$,
$b\in\mathbb{R}$ represents shear strength (amplitude-phase coupling parameter when writing the deterministic part of (\ref{NormalForm}) in polar coordinates), $a > 0$, $\beta\in \mathbb{R}$, and $W_t^1, W_t^2$ denote independent one-dimensional Brownian motions.


In the absence of noise ($\sigma=0$), the stochastic differential equation \eqref{NormalForm} is a normal form for the supercritical Hopf bifurcation:
when $\alpha \leq 0$ the system has a globally attracting equilibrium at $(x,y) = (0,0)$ which is exponentially stable until $\alpha=0$ and, when $\alpha >0$, the system has a limit cycle at $\left\{(x,y) \in \mathbb{R}^2 \,:\, x^2 + y^2 = \alpha/a\right\}$ which is globally attracting on $\mathbb{R}^{2}\setminus{\{0\}}$.

In the presence of noise ($\sigma \neq 0$), statistical information about the (one point) dynamics of \eqref{NormalForm} can be described by the Fokker--Planck equation and its stationary density. In this case, the stationary density can be calculated analytically, yielding
\begin{equation}\label{StatDens1}
\frac{ 2 \sqrt{2 a} }{\sqrt{\pi}\sigma \erfc (- \alpha/\sqrt{2 a \sigma^2}   )}\exp\left(\frac{2\alpha(x^2+y^2)-a(x^2+y^2)^2}{2\sigma^2}\right)\,.
\end{equation}
We note in particular
that this density
does not depend on the shear parameter $b$.

We observe a clear relation between the stationary measures in the presence of noise ($\sigma>0$) and the attractors in the deterministic limit: the stationary density is maximal on attractors of the deterministic limit dynamics and (locally) minimal on its repellers, see Figure~\ref{FigHopfStat}.
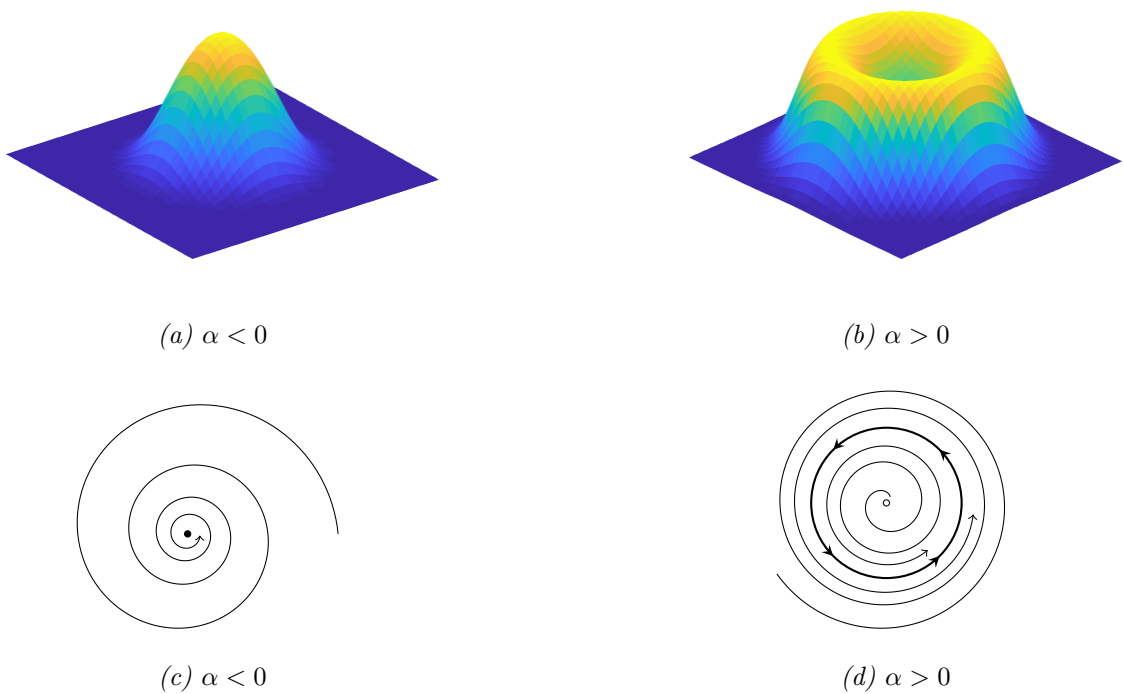
\begin{figure}[ht!]
\centering
\begin{subfigure}[b]{.45\textwidth}
  \centering
  \includegraphics[width=1\linewidth]{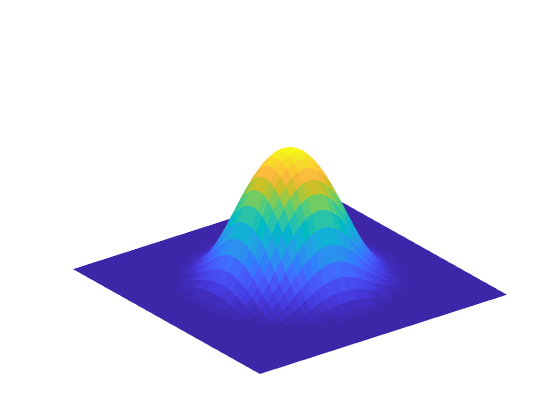}
  \caption{$\alpha < 0$}
   \vspace{1.0\baselineskip}
  \label{fig1_1}
\end{subfigure}%
\hfill
\begin{subfigure}[b]{.45\textwidth}
  \centering
  \includegraphics[width=1\linewidth]{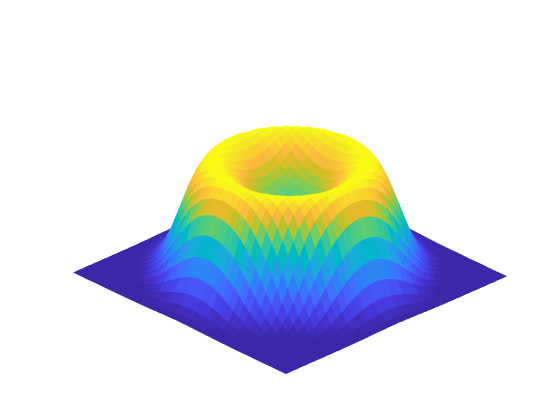}
  \caption{$\alpha  >0$}
  \vspace{1.0\baselineskip}
  \label{fig1_4}
  \end{subfigure}
\hfill
\begin{subfigure}[b]{.45\textwidth}
\centering
\begin{tikzpicture}[scale=1]
\draw[fill]
(0,0) circle (.25ex);
\draw [->,domain=0:25,variable=\t,samples=500]
        plot ({\t r}: {2*exp(-0.1*\t)});
\end{tikzpicture}
 \vspace{0.5\baselineskip}
  \caption{$\alpha < 0$}
  \label{det_-1}
\end{subfigure}%
\hfill
\begin{subfigure}[b]{.45\textwidth}
  \centering
  \begin{tikzpicture}[scale=1]
 \draw
(0,0) circle (.25ex);
\draw [->,domain=10:25,variable=\t,samples=500]
        plot ({\t r}: {1+2*exp(-0.1*\t)});
\draw [->,domain=1:18,variable=\t,samples=500]
        plot ({\t r}: {1-exp(-0.1*\t)});
\path [draw, thick, postaction={on each segment={mid arrow}}]
(0,0) circle [radius =1];
\end{tikzpicture}
 \vspace{0.5\baselineskip}
  \caption{$\alpha > 0$}
  \label{det_1}
\end{subfigure}%
\caption{\label{FigHopfStat} Shape of the stationary density of (\ref{NormalForm}) with noise and corresponding phase portraits of the deterministic limit. The qualitative features only depend on the sign of
the linear stability parameter $\alpha$. Figures (a) and (b) present the shapes of the stationary densities in the presence of noise. (a) is charactised by a unique maximum at the origin and (b) by a local minimum at the origin surrounded by a circle of maxima when $\alpha>0$. Figures (c) and (d) show phase portraits in the determinstic limit $\sigma=0$ displaying an attracting equilibrium if $\alpha<0$ and an attracting limit cycle if $\alpha>0$, precisely where stationary densities have their maxima.}
\end{figure}

From Figure~\ref{FigHopfStat} it is natural to propose that the stochastic differential equation (\ref{NormalForm}) has a bifurcation at $\alpha=0$, represented by the qualitative change of the shape of the stationary density.  Such kind of bifurcation is called a \emph{phenomenological bifurcation}, cf.~\cite{a98}.

In this paper, we consider the system (\ref{NormalForm}) with noise from a random dynamical systems point of view: with a canonical model for the noise,  (\ref{NormalForm}) can be represented as a dynamical system that is driven by a random signal.

While the stationary density (\ref{StatDens1}) provides certain statistics about the dynamics of (\ref{NormalForm}), by the fact that the underlying Markov process only models probabilistically a single time-series,
many relevant dynamical properties cannot be captured, such as a comparison of the trajectories of nearby initial conditions (with the same noise).

As trajectories of random dynamical systems depend on the noise realisation, one does not a priori expect any asymptotic long-term convergent behaviour of individual trajectories to a fixed attractor. An alternative view point that circumvents this problem and often yields convergence, is to consider, for a fixed noise realisation in the past, the flow of a set of initial conditions from time $t=-T$ to a fixed endpoint in time, say $t=0$, and then take the (pullback) limit $T\to\infty$. If trajectories of initial conditions converge under this procedure to some set, then this set is called a \emph{pullback attractor}. To illustrate the pullback dynamics of \eqref{NormalForm}, in Figure~\ref{fig:chaosNF}, we present some numerical examples\footnote{The simulations in this paper are based on an explicit Euler--Maruyama integration of the stochastic differential equation~\eqref{NormalForm}, usually with time step size $10^{-3}$. When we compute Lyapunov exponents, we use an explicit second-order Runge--Kutta method for integrating the variational equation.}.
We observe two distinctly different behaviours: either all initial conditions converge to a fixed point, see (a)--(d), or all initial conditions converge to a rather complicated object, see (e)--(h). The former is indicative of the phenomenon of  \emph{synchronisation}, i.e.~convergence of all trajectories to a single random equilibrium point, while the latter points to a random strange attractor.

\begin{figure}[H]
\centering
\begin{subfigure}[b]{.25\textwidth}
  \centering
  \includegraphics[width=1\linewidth]{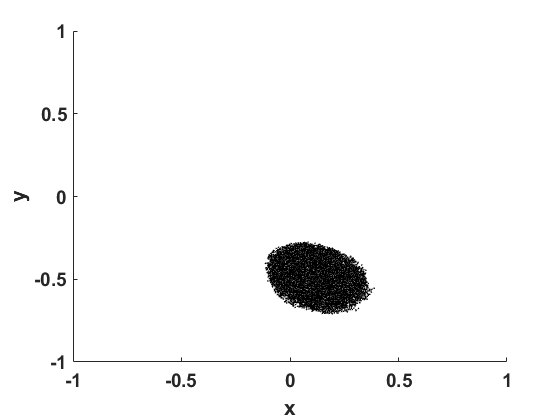}
  \caption{$\alpha=-1,b=1,T=5$}
  \label{synchro2_nf_alpha-1.png}
\end{subfigure}%
\begin{subfigure}[b]{.25\textwidth}
  \centering
  \includegraphics[width=1\linewidth]{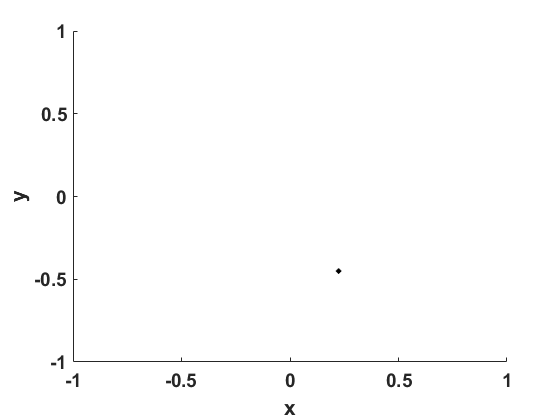}
  \caption{$\alpha=-1,b=1,T=50$}
  \label{synchro3_nf_alpha-1.png}
\end{subfigure}%
\begin{subfigure}[b]{.25\textwidth}
  \centering
  \includegraphics[width=1\linewidth]{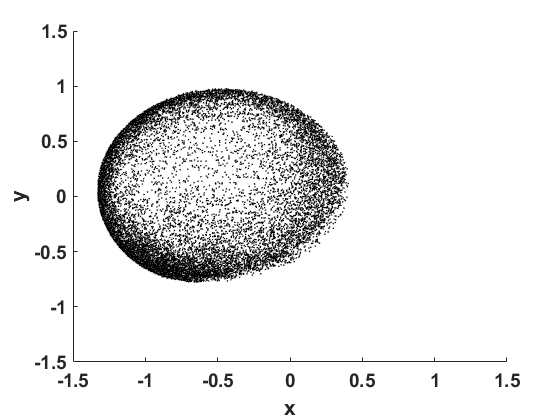}
  \caption{$\alpha=1,b=1,T=5$}
  \label{synchro2_nf}
\end{subfigure}%
\begin{subfigure}[b]{.25\textwidth}
  \centering
  \includegraphics[width=1\linewidth]{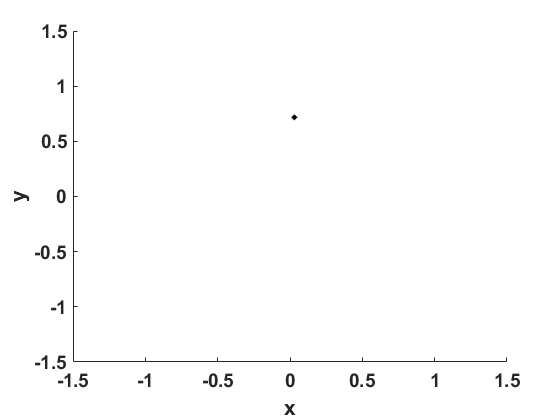}
  \caption{$\alpha=1,b=1,T=50$}
  \label{synchro4_nf}
\end{subfigure}%
\hfill
\begin{subfigure}{.25\textwidth}
  \centering
  \includegraphics[width=1\linewidth]{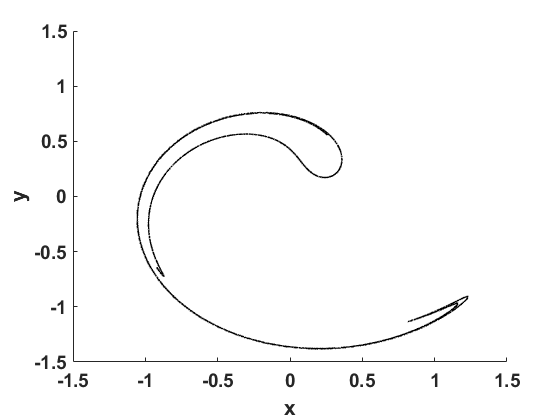}
  \caption{$\alpha=-1,b=20,T=5$}
  \label{strange2_nf_alpha-1.png}
\end{subfigure}%
\begin{subfigure}{.25\textwidth}
  \centering
  \includegraphics[width=1\linewidth]{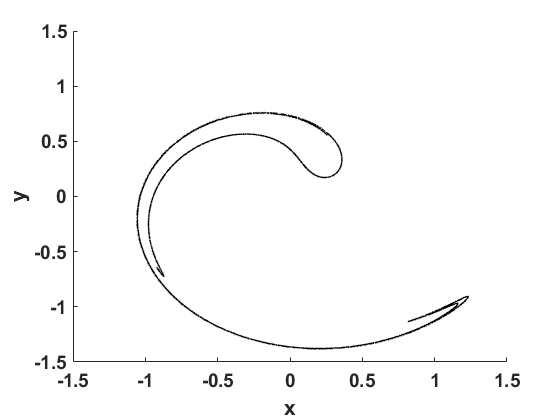}
  \caption{$\alpha=-1,b=20,T=50$}
  \label{strange3_nf_alpha-1.png}
\end{subfigure}%
\hfill
\begin{subfigure}{.25\textwidth}
  \centering
  \includegraphics[width=1\linewidth]{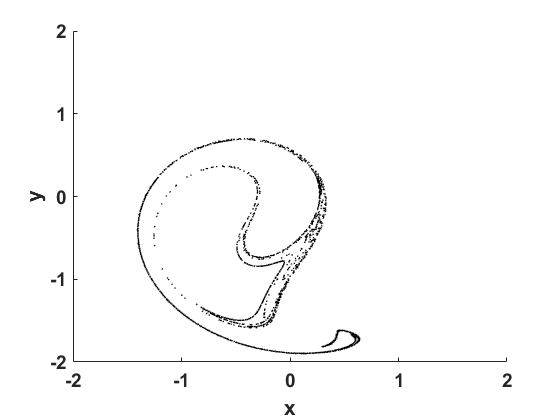}
  \caption{$\alpha=1,b=8,T=5$}
  \label{strange2_nf}
\end{subfigure}%
\begin{subfigure}{.25\textwidth}
  \centering
  \includegraphics[width=1\linewidth]{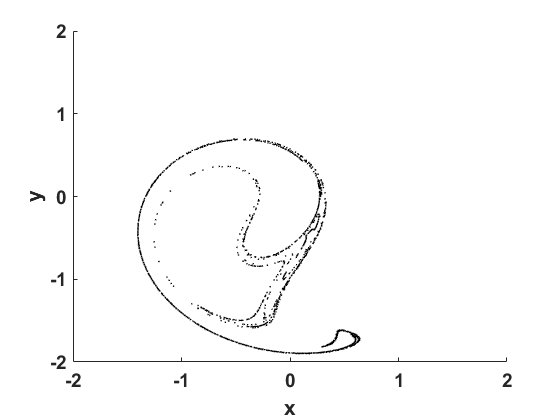}
  \caption{$\alpha=1,b=8,T=50$}
  \label{strange4_nf}
\end{subfigure}%
\caption{
Pullback dynamics of \eqref{NormalForm} with $\sigma=\beta=a=1$ for initial conditions chosen in approximation of the stationary density. In (a)--(d), in the presence of small shear we observe synchronisation, i.e.~pullback convergence of all trajectories to a single point, irrespective of the linear stability at the origin. In (e)--(h), in the presence of sufficiently large shear there is no synchronisation but pullback convergence to a more complicated object (random strange attractor), again irrespective of the linear stability at the origin.
}
\label{fig:chaosNF}
\end{figure}
The differences between the types of pullback attractor can also be observed from the Lyapunov exponents, representing the asymptotic long-term average derivative along trajectories. Roughly speaking, random attractors with negative Lyapunov exponents are associated with synchronisation and a positive Lyapunov exponent impedes synchronisation. Accordingly, in Figure~\ref{fig:chaosNF} (a)--(d) we have negative Lyapunov exponents and in (e)--(h) the largest Lyapunov exponent is positive. In Figure~\ref{fig:posLyap} we present a numerical investigation of the top Lyapunov exponent as a function of the relevant parameters. We note that, in contrast with the deterministic and statistical transitions at $\alpha=0$, the change of sign of the top Lyapunov exponent is indicative of a \emph{dynamical bifurcation}, cf.~\cite{a98}, arises along an altogether different curve in the $(b,\alpha)$ phase diagram. In particular, we note that as the stationary density is independent of $b$, different dynamical behaviours underly identical stationary measures, reconfirming our earlier claim that the one-point Markov process and associated stationary measure only provide partial information about the dynamics of a random dynamical system.
\begin{figure}[H]
\centering
  \includegraphics[width=0.5\linewidth]{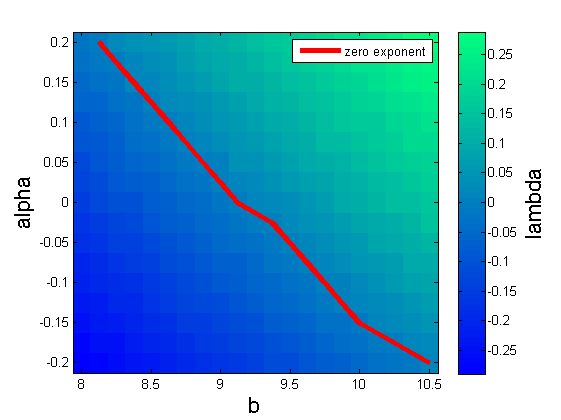}
\caption{Numerical approximation of the top Lyapunov exponent for system~\eqref{NormalForm} as a function of the linear stability at the  origin ($\alpha$) and strength of shear ($b$) with $a= \beta = \sigma =1$.
The red curve highlights the border between regions with negative and positive top Lyapunov exponents, corresponding to synchronisation or random strange attractor, respectively.
}
\label{fig:posLyap}
\end{figure}

Finally, we address a more subtle differentiation between two types of synchronisation that may arise. Synchronisation may be \emph{uniform}, so that trajectories are guaranteed to approximate each other bounded by upper estimates that are independent of the noise realisation, or \emph{non-uniform}, when such uniform upper estimates do not exist. In the latter case, the time it takes for two trajectories to converge up to a certain given margin is bounded for any fixed noise realisation, but assessed over all noise realisations these bounds have no maximum.
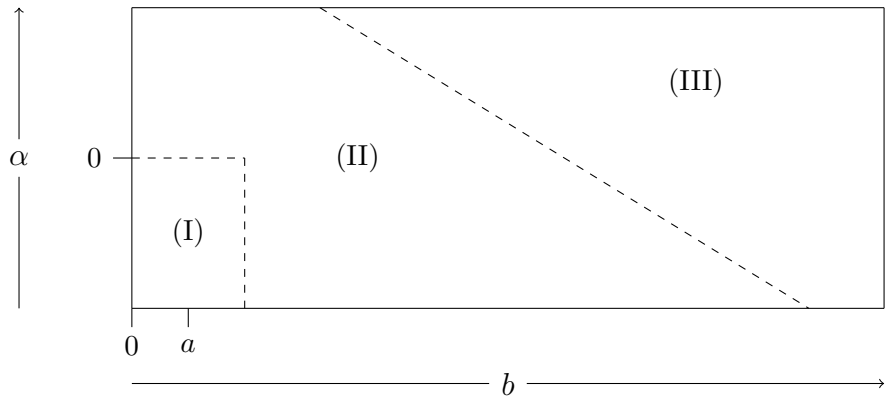
\begin{figure}[H]
\vspace{1.0\baselineskip}
\centering
\begin{tikzpicture}[scale =1]
    \draw (-5,2)--(5,2);
    \draw (-5,-2)--(5,-2);
    \draw  (-5,-2) -- (-5,2);
    \draw  (5,-2) -- (5,2);
    \draw[dashed] (-5,0)--(-3.5,0);
    \draw[dashed] (-3.5,0)--(-3.5,-2);
    \node at (-4.25,-1) {(I)};
    \node at (-2,0) { (II)};
    \draw[dashed] (-2.5,2)--(4,-2);
    \node at (2.5,1) { (III)};
    \draw (-4.25,-2)--(-4.25,-2.25);
    \node at (-4.25,-2.5) { $a$};
    \draw (-5,0)--(-5.25,0);
    \node at (-5.5,0) { $0$};
    \draw (-5,-2)--(-5,-2.25);
    \node at (-5,-2.5) { $0$};
    \node at (-6.5,0) {\large $\alpha$};
    \node at (0,-3) {\large $b$};
    \draw  (-5,-3) -- (-0.25,-3);
    \draw[->] (0.25,-3) -- (5,-3);
    \draw  (-6.5,-2) -- (-6.5,-0.25);
    \draw[->] (-6.5,0.25) -- (-6.5,2);
  \end{tikzpicture}
  \caption{For $a,\beta, \sigma$ fixed, we partition the $(b,\alpha)$-parameter space associated with \eqref{NormalForm} into three parts with different stability behaviour. Region (I) represents uniform synchronisation, only possible for non-positive $\alpha$ and small $b$. In region (II), we observe non-uniform synchronisation, i.e.~finite-time instabilities occur, but the asymptotic behaviour is exponentially stable for almost all trajectories. (The border between (I) and (II) is described in Theorems E and F.) Region (III) exhibits a positive top Lyapunov exponent and the absence of synchronisation since the shear is large enough for locally unstable behaviour to prevail (cf.~Conjecture~D).}
\label{fig:phasediag}
\end{figure}
\begin{figure}[H]
\centering
\begin{subfigure}{.4\textwidth}
  \centering
  \includegraphics[width=1\linewidth]{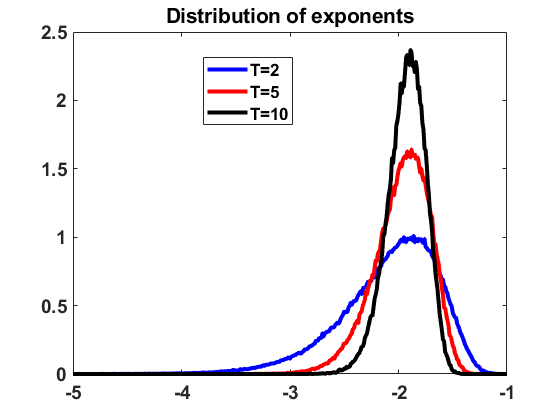}
  \caption{$b=1, \alpha =-1$}
  \label{b1alpha-1}
\end{subfigure}%
\begin{subfigure}{.4\textwidth}
  \centering
  \includegraphics[width=1\linewidth]{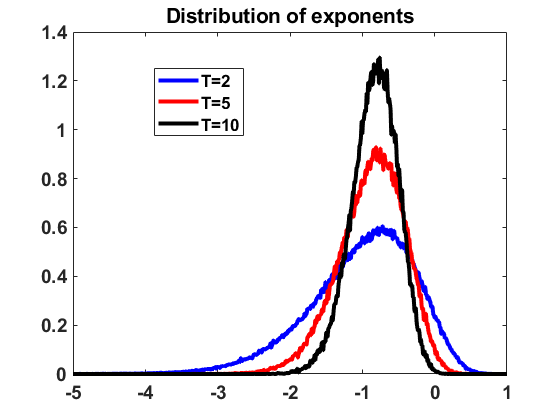}
  \caption{$b=1, \alpha=1$}
  \label{b1alpha1}
\end{subfigure}%
\hfill
\begin{subfigure}{.4\textwidth}
  \centering
  \includegraphics[width=1\linewidth]{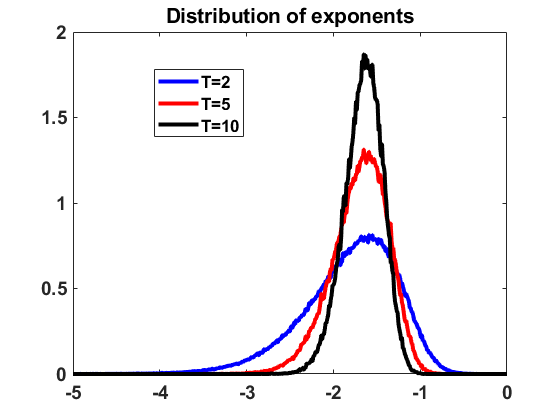}
  \caption{$b=3, \alpha =-1$}
  \label{b3alpha-1}
\end{subfigure}%
\begin{subfigure}{.4\textwidth}
  \centering
  \includegraphics[width=1\linewidth]{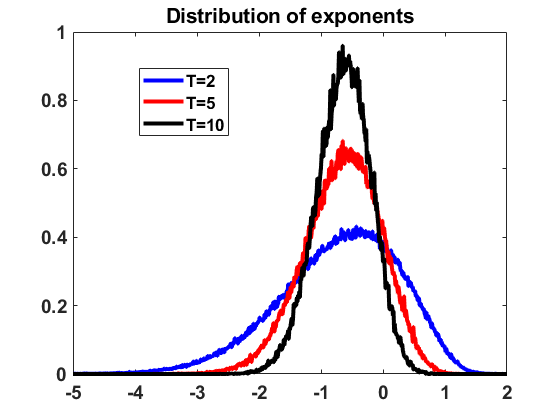}
  \caption{$b =3, \alpha=1$}
  \label{b3alpha1}
\end{subfigure}%
\hfill
\begin{subfigure}{.4\textwidth}
  \centering
  \includegraphics[width=1\linewidth]{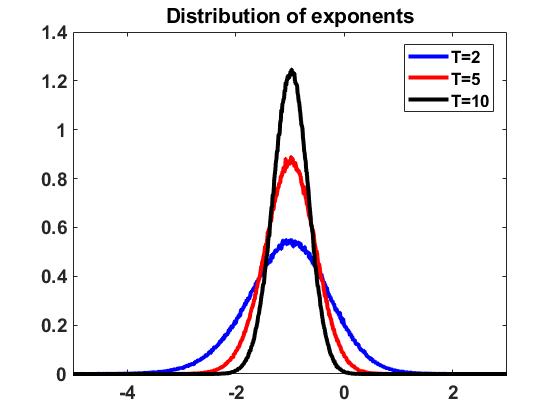}
  \caption{$b =8, \alpha =-1$}
  \label{b8alpha-1}
\end{subfigure}%
\begin{subfigure}{.4\textwidth}
  \centering
  \includegraphics[width=1\linewidth]{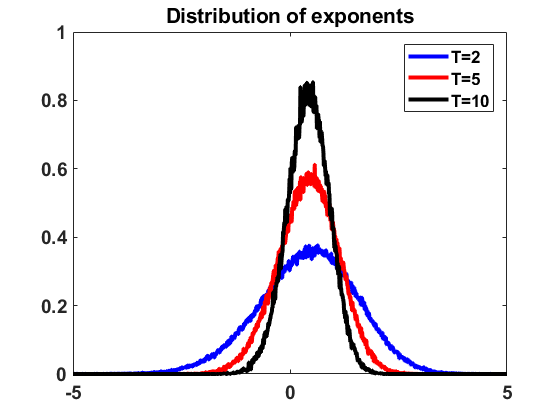}
  \caption{$b=8, \alpha=1$}
  \label{b8alpha1}
\end{subfigure}%
\caption{Distribution of finite-time ($T>0$) Lyapunov exponents of (\ref{NormalForm}) with $a=\beta=\sigma =1$,  $\alpha \in\{-1,1\}$, $T\in\{2,5,10\}$ and $b\in\{1,3,8\}$, illustrating the type of distributions in phases (I) uniform synchronisation (a), (c); (II) non-uniform synchronisation (b), (d), (e); and (III) absence of synchronisation (f).}
\label{finitetimedist}
\end{figure}
It turns out that the uniformity of the synchronisation is related with the distribution of \emph{finite-time Lyapunov exponents}, reflecting the average derivatives along trajectories for finite time. The (unique) top Lyapunov exponent of an attractor is associated with the limit of the distribution of finite-time Lyapunov exponents as the time over which derivatives are averaged goes to infinity. Importantly, while this distribution converges to a Dirac measure concentrated in the top Lyapunov exponent, the support of this distribution typically converges to a wider range. If this range is contained entirely within the negative real axis, synchronisation is uniform. But it may also happen that the top Lyapunov is negative while the limit of the support of finite-time Lyapunov exponents extends into the positive half line, which results in non-uniform synchronisation. In Figure~\ref{finitetimedist}, these scenarios are illustrated with numerical computations. It is natural to find an interface with non-uniform synchronisation in the phase diagram between uniform synchronisation regions and regions without synchronisation. For a sketch of the corresponding phase diagram regions for (\ref{NormalForm}), see Figure~\ref{fig:phasediag}.

The main aim of this paper is to provide a precise mathematical analysis to describe and explain the observations sketched above.
Numerical investigations by Lin and Young \cite{ly08}, Wieczorek \cite{w09} and Deville {\em et al} \cite{dsr11} already highlighted the fact that shear can
cause Lyapunov exponents to become positive and induce chaotic behaviour.  A first analytical proof of this phenomenon has been given by us in \cite{elr16} in the case of a stochastically driven limit cycle on the cylinder. A prove of shear-induced chaos with periodic driving (kicks) was given before by Wang and Young \cite{wy03}. The stability properties of the stochastic system for small noise limits and small shear have been studied in \cite{dsr11} as an example of a non-Hamiltonian system perturbed by noise. The authors also conjectured asymptotic instabilities represented by a positive top Lyapunov exponent. They did not prove implications for the associated random dynamical system in terms of its random attractor and invariant measure which is the subject of the first part of this paper.

%
%

We first establish (Theorem A) that the stochastic differential equation \eqref{NormalForm} induces a random dynamical system and possesses a random attractor for all choices of parameters. Using results from \cite{fgs16}, we show that a negative top Lyapunov exponent implies the random attractor being a random equilibrium.  We then prove (Theorem B) the synchronisation of almost all trajectories from all initial conditions in forward time with exponential speed. We also achieve an explicit upper bound for the shear as a function of other parameters for having a negative top Lyapunov exponent (Theorem C), extending results in \cite{dsr11} to the full parameter space.

%
%
We finally assert (Conjecture D) the appearance of a positive top Lyapunov exponent beyond a critical shear levels for any given value of $\alpha$, cf.~Figure~\ref{fig:posLyap}. This would in turn imply the existence of a random strange attractor with positive entropy and SRB sample measures. Based on numerical evidence, we conjecture this scenario also for negative $\alpha$ which is remarkable in view of the fact that in the literature shear-induced chaos is associated with random perturbations of limit cycles and not equilibria.


%
The second part of this paper focuses on parameter-dependence of finite-time Lyapunov exponents and uniform attractivity and the dichotomy spectrum associated with the linear random dynamical system on the tangent space along trajectories. In the case of small shear, we establish (Theorem E) the existence of a bifurcation at the deterministic Hopf parameter value $\alpha=0$ from a global uniformally attractive random equilibrium ($\alpha<0$) to a non-uniformly attractive random equilibrium ($\alpha>0$). This bifurcation is accompanied by the emergence of positive finite-time Lyapunov exponents and a loss of hyperbolicity of the associated \emph{dichotomy spectrum} $\Sigma= [- \infty, \alpha]$. 
This result provides an example of the bifurcation scenario proposed in \cite{cdlr16}, highlighting the importance of new notions of bifurcation to complement the deterministic ones, by showing that despite the persistence of random equilibria, additive noise does not necessarily "destroy" bifurcations, cf.~\cite{cf98}.

Finally, we establish (Theorem F) the existence of positive finite-time Lyapunov exponents. In particular, we show that for any $\alpha \in \mathbb{R}$, there exist arbitrarily large finite-time Lyapunov exponents for sufficiently strong shear intensity $b$. This is the first analytical result on shear-induced chaos in \eqref{NormalForm}. It is in general challenging to obtain lower bounds for the top Lyapunov exponent in dimension larger than one due to the subadditivity property of matrices, cf.~\cite{y08}. Therefore analytical results on positive Lyapunov exponents for random dynamical systems have only been achieved in certain special cases, like in simple time-discrete models \cite{ls12}, certain linear models \cite{elr16} and under special circumstances enabling for stochastic averaging arguments \cite{b04, bg02}. It remains an open problem to prove Conjecture D.

The results of this paper are part of an emerging bifurcation theory for random dynamical systems. Earlier attempts to develop such a theory (notably by Ludwig Arnold, Peter Baxendale and coworkers \cite{a98, ass96, b94, sh96} in the 1990s) resulted in notions of so-called \emph{phenomenological} (or "P") bifurcations and \emph{dynamical} (or "D") bifurcations, but our research and that of others \cite{acd16,cdlr16,drk15,Lamb_15_1,Zmarrou_07_1} suggest that these do not comprehensively capture the intricacies of bifurcation in random dynamical systems. Despite its relevance for many applications of topical interest, a bifurcation theory of random dynamical systems is still in its infancy and often remains restricted to the detailed analysis of relatively elementary examples. Many studies in the context of stochastic Hopf bifurcation have considered the Duffing--van der Pol oscillator with multiplicative white noise \cite{ass96, sh96, sh962}. Few rigorous results have been obtained and most studies only led to conjectures based on numerical observations \cite{ko99}. Our model is exemplary in the following sense: firstly, it discusses the typical phenomenon of random systems to exhibit a transition between synchronisation and chaos. Secondly, the normal form is locally equivalent to that of a generic deterministic Hopf bifurcation and, hence, at least for small noise, we can expect other examples of Hopf bifurcation to feature similar behaviour.

This paper is organised as follows. Section~\ref{intro} comprehensively introduces the technical framework and formulates the main results of this paper. Section~\ref{Generation} is dedicated to a detailed of proof of Theorem A, establishing the existence of a random attractor for all parameters. In Section~\ref{Hopf} we prove Theorems B and C and show some statistical properties of the random equilibrium. In conclusion, Section~\ref{Random Hopf} contains the proofs of Theorems E and F highlighting different aspects of the random bifurcations in $\alpha$ and $b$. We also provide an Appendix with background material on random dynamical systems comprising the most relevant definitions and results used in this paper.

\section{Statement of the main results} \label{intro}
The stochastic differential equation \eqref{NormalForm} can be rewritten as
\begin{equation}\label{MainEq}
  \rmd Z_t= f(Z_t)\rmd t
  +
  \sigma \rmd W_t\,,
\end{equation}
where $Z_t=(x_t,y_t)^\top$ and $W_t=(\rmd W^1_t, \rmd W^2_t)^\top$, and the function $f:\mathbb R^2\rightarrow \mathbb R^2$ is defined by
\[
f(Z):=
\left(
\begin{array}{ll}
\alpha & -\beta\\
\beta & \alpha
\end{array}
\right) Z- (x^2+y^2) \left(
\begin{array}{ll}
a & -b\\
b & a
\end{array}
\right)Z\,.
\]
To investigate sample path properties of the solutions of \eqref{NormalForm}, it is convenient to work with the canonical sample path space of Brownian motions. Let $\Omega=C_0(\mathbb R,\mathbb R^2)$ be the space of all continuous functions $\omega:\mathbb R\rightarrow \mathbb R^2$ satisfying that $\omega(0)=0$.
We endow $\Omega$ with the compact open topology and denote by $\mathcal F=\mathcal B(\Omega)$ the Borel $\sigma$-algebra on $\Omega$.

It is well known that there exits the so-called \emph{Wiener probability measure} $\mathbb P$ on $(\Omega,\mathcal F)$ which ensures that the two process $(W_t^1)_{t\in\mathbb R}$ and $(W_t^2)_{t\in\mathbb R}$, defined by $(W_t^1(\omega), W_t^2(\omega))^{\top}:=\omega(t)$ for $\omega\in\Omega$, are independent one-dimensional Brownian motions. We define the sub $\sigma$-algebra $\mathcal{F}_{s,t}$ as the $\sigma$-algebra generated by $\omega(u)- \omega(v)$ for $s \leq v \leq u \leq t$. For each $t\in\mathbb R$, we define the shift map $\theta_t:\Omega\rightarrow \Omega$ by
\begin{equation*}
(\theta_t \omega)(s) = \omega(s+t) - \omega(t) \fa s \in \R\,.
\end{equation*}
It is well known that $(\theta_t)_{t\in\mathbb R}$ is an ergodic flow preserving the probability measure $\mathbb P$, see e.g.~\cite{a98}. Thus,  $(\Omega,\mathcal F,\mathbb P,(\theta_t)_{t\in \mathbb R})$ is an ergodic dynamical system.
\subsection{Generation of a random dynamical system with a random attractor}

Given $\omega\in\Omega$, an initial value $Z\in\mathbb R^2$ and $T>0$, we say that a continuous function $\varphi(\cdot,\omega,Z):[0,T]\rightarrow \mathbb R^2$ solves the stochastic differential equation \eqref{NormalForm} if it satisfies the integral equation
\[
\varphi(t,\omega,Z)
=
Z
+
 \int_0^t f(\varphi(s,\omega,Z))\,\rmd s
+
\sigma \omega(t)\fa t\in [0,T]\,.
\]
The first result in this paper concerns global existence of solutions of \eqref{NormalForm} for almost every sample path, implying that the solutions do not blow up in forward time. We show that the solutions of \eqref{NormalForm} generate a random dynamical system $(\theta, \varphi)$ (see \cite[Definition 1.1.1]{a98} for a general definition). This means that the ($\mathcal{B}(\mathbb{R}_0^+)\otimes \mathcal{F} \otimes \mathcal{B}(\mathbb R^2)$, $\mathcal{B}(\mathbb R^2)$)-measurable mapping $\varphi: \mathbb{R}_0^+ \times \Omega \times \mathbb R^2 \to \mathbb R^2, (t, \omega, x) \mapsto \varphi(t, \omega,x)$, is a cocycle over $\theta$, i.e.
\[
  \varphi(0, \omega, \cdot) \equiv \Id \quad \text{and} \quad \varphi(t+s, \omega,x) = \varphi(t, \theta_s \omega,\varphi(s, \omega,x)) \fa \omega \in \Omega, x \in \R^2 \text{ and } t, s \ge 0\,.
\]
In addition to the generation of a random dynamical system, the following theorem addresses also the existence of a random attractor (see Appendix~\ref{sec-randomattr} for a definition).

\begin{theorem-nonA}[Generation of a random dynamical system with a random attractor]
  For the stochastic differential equation \eqref{NormalForm}, there exists a $\theta$-invariant $\mathcal F$-measurable set $\widehat\Omega\subset \Omega$ of full probability  such that the following statements hold.
  \begin{itemize}
  \item [(i)]  For all $\omega\in\widehat \Omega$ and $Z\in\mathbb R^2$, the stochastic differential equation \eqref{NormalForm} admits a unique solution $\varphi(\cdot,\omega,Z)$ such that $\varphi$ forms a cocycle for  a random dynamical system on $(\widehat \Omega,\mathcal F,\mathbb P, (\theta_t)_{t\in\mathbb R})$.
  \item [(ii)] There exists a random attractor $A\in \mathcal F\otimes \mathcal B(\R^2)$ of the random dynamical system $(\theta,\varphi)$ such that $\omega \mapsto A(\omega)$ is measurable with respect to $\mathcal{F}_{-\infty}^{0}$, i.e.~the past of the system.
  \end{itemize}
\end{theorem-nonA}

Since the difference of the spaces $\Omega$ and $\widehat \Omega$ is a set of measure zero, we identify both in the following.

\subsection{Negativity of top Lyapunov exponent and synchronisation}

The following results concern the asymptotic behaviour of trajectories, in particular their stability properties. This will give information about the structure of the random attractor $A$ associated with the stochastic differential equation~\eqref{NormalForm}.

To analyse asymptotic stability, we study the linearisation $\Phi(t,\omega,Z):=\rmD_x\varphi(t,\omega,Z)$. A direct computation yields that $\Phi(0,\omega,Z)=\Id$ and
\begin{equation}\label{Linearization}
\dot\Phi(t,\omega,Z)=
\rmD f(\varphi(t,\omega,Z))\Phi(t,\omega,Z)\,.
\end{equation}
It is easy to observe that $\Phi$ is a linear cocycle over the skew product flow $(\Theta_t)_{t \in\R^+_0}$ on $\Omega\times \mathbb R^2$, defined by
\[
\Theta_t(\omega,Z):= (\theta_t\omega, \varphi(t,\omega,Z))\,.
\]
In fact,  $(\Theta, \Phi)$ is a linear random dynamical system, where the ergodic dynamical system $(\theta_t)_{t\in\R}$ is replaced by $(\Theta_t)_{t \in\R^+_0}$. We obtain an ergodic probability measure for the skew product flow $(\Theta_t)_{t \in\R^+_0}$ by using the fact that there exists a one-to-one correspondence between the stationary measure $\rho$ for the Markov semigroup associated to \eqref{MainEq} and a certain invariant measure of $(\Theta_t)_{t \in\R^+_0}$.

In more detail, recall from \eqref{StatDens1} that the density of the unique stationary distribution $\rho$  reads as
\begin{equation} \label{StatDens}
  p(x,y) = K_{a,\alpha,\sigma}  \exp \left( \frac{2 \alpha (x^2 + y^2) - a(x^2 + y^2)^2}{2 \sigma^2} \right) \fa (x,y)\in\R^2\,,
\end{equation}
where $K_{a,\alpha,\sigma} >0$ is the normalisation constant and is given by
\begin{equation*}
K_{a,\alpha,\sigma} = \frac{ 2 \sqrt{2 a} }{\sqrt{\pi}\sigma \erfc \left(- \frac{\alpha}{\sqrt{2 a \sigma^2}}   \right)}\,.
\end{equation*}
The stationary measure $\rho$ gives rise to an invariant measure $\mu$ for $(\Theta_t)_{t \in\R^+_0}$ on $\Omega\times \mathbb R^2$ in the following sense: the push-forward limit
\begin{equation*}
  \mu_{\omega} := \lim_{t \to \infty} \varphi(t, \theta_{-t}\omega)\rho
\end{equation*}
exists for almost all $\omega\in\Omega$ and is an $\mathcal{F}_{-\infty}^{0}$-measurable random measure, i.e.~$\omega\mapsto \mu_\omega(B)$ is $\mathcal{F}_{-\infty}^{0}$-measurable for any $B\in\mathcal B(\R^2)$.
This defines a \emph{Markov measure} $\mu$ on $(\Omega\times\mathbb R^2,\mathcal F\otimes \mathcal B(\mathbb R^2))$ via
\[
\mu(C)
:=\int_{\Omega}
\mu_{\omega}(C_{\omega})\; \rmd \mathbb P(\omega) \fa C\in \mathcal F\otimes \mathcal B(\mathbb R^2)\,,
\]
where $C_{\omega}:=\{Z\in\mathbb R^2: (\omega,Z)\in C\}$. $\mu$ is invariant under $(\Theta_t)_{t \in\R^+_0}$ (see e.g.~\cite{c91}).
Reversely, the stationary measure $\rho$ is given by
\begin{equation} \label{Expectedmeasure}
\rho(B) = \int_{\Omega}
\mu_{\omega}(B)\; \rmd \mathbb P(\omega) \fa  B \in \mathcal{B}(\mathbb{R}^2)\,.
\end{equation}
The uniqueness of the stationary measure $\rho$ with density $p(x,y)$ implies that the invariant measure~$\mu$ is ergodic. We will see in Proposition \ref{Singularity} that the linear system $\Phi$ defined in \eqref{Linearization} satisfies the integrability condition
\[
\sup_{0 \leq t \leq 1} \ln^+ \| \Phi(t, \omega,Z) \| \in L^1(\mu)\,.
\]
Therefore, we can apply Oseledets' Multiplicative Ergodic Theorem (see Appendix~\ref{LyapSpec}) to obtain the Lyapunov spectrum of the linear random dynamical system $(\Theta, \Phi)$. In particular, the top Lyapunov exponent is given by
\begin{equation}\label{TopLE}
\lambda_{\mathrm{top}}
:=
\lim_{t \to \infty}
\frac{1}{t} \ln\|\Phi(t,\omega,Z)\|\quad \hbox{for } \mu\hbox{-almost all } (\omega,Z)\in\Omega\times \mathbb R^2\,.
\end{equation}
The top Lyapunov exponent allows to characterise synchronisation for the random dynamical system generated by \eqref{NormalForm}, i.e.~if for all $Z_1,Z_2\in\mathbb R^2$, we have
\[
\lim_{t \to \infty}\|\varphi(t,\omega,Z_1)-\varphi(t,\omega,Z_2)\|=0\faa \omega\in\Omega\,.
\]
\begin{theorem-nonB}[Existence of random equilibrium and synchronisation of trajectories]
  Suppose that $\lambda_{\mathrm {top}} <0$. Then the random attractor $A$ for the stochastic differential equation~(\ref{NormalForm}) is given by a random equilibrium, i.e.~$A(\omega)$ is a singleton for almost all $\omega\in\Omega$. In addition, the the stochastic differential equation~(\ref{NormalForm}) admits exponentially fast synchronisation, i.e.~for all $Z_1,Z_2\in\mathbb R^2$, we have
  \[
  \limsup_{t \to \infty}
  \frac{1}{t} \ln
  \|
  \varphi(t,\omega,Z_1)-\varphi(t,\omega,Z_2)\|<0\faa\omega\in\Omega\,.
  \]
\end{theorem-nonB}

We now aim to determine the region of parameters for which $\lambda_{\mathrm {top}}<0$.  In \cite{dsr11}, analytical results are obtained that show that $\lambda_{\mathrm{top}}$ is negative in certain regions of the parameters space, in particular when shear is small. The following theorem extends this result to a larger region in the parameter space.
\begin{theorem-nonC}[Small shear implies synchronisation]
  For each $a, \alpha,\beta,\sigma$, let
  \[
  \kappa:=
  a\sqrt{\frac{\pi K_{a,\alpha,\sigma} \sigma^2}{\alpha+\pi K_{a,\alpha,\sigma} \sigma^2}\left(\frac{\pi K_{a,\alpha,\sigma} \sigma^2}{\alpha+\pi K_{a,\alpha,\sigma} \sigma^2}+2\right)}.
  \]
  Then the top Lyapunov exponent $\lambda_{\mathrm{top}}$ is negative if $|b|\leq \kappa$.
\end{theorem-nonC}
\begin{remark}
(i) Note that $K_{a,0,\sigma}=\frac{2\sqrt{2a}}{\sqrt \pi \sigma}$, and Theorem C then implies that $\lambda_{\mathrm{top}}<0$ provided that $|b|<\sqrt{3}a$ and $\alpha$ is sufficiently small. This special case is considered in \cite[Proposition 4.1]{dsr11}.

\noindent
(ii) For fixed $a$ and $\alpha$, we have
\[
\lim_{\sigma\to\infty}\pi K_{a,\alpha,\sigma}\sigma^2=\lim_{\sigma\to\infty}\frac{2\sqrt{2\pi a}\sigma }{\erfc \left(- \frac{\alpha}{\sqrt{2 a \sigma^2}}   \right)}=\infty\,.
\]
Therefore, by Theorem C we have  $\lambda_{\mathrm{top}}<0$ provided that $|b|<\sqrt{3}a$ and the noise intensity $\sigma$ is sufficiently large.
\end{remark}

Numerical evidence from \cite{dsr11} and Figure~\ref{fig:posLyap} suggest that large shear leads to positive top Lyapunov exponent. Unfortunately, we are not able to prove this analytically and formulate this in the following conjecture. Note that in \cite{elr16}, positivity of the top Lyapunov exponent was analytically established for a two-dimensional system that admits large shear.
\begin{theorem-nonD}[Large shear induces chaos]
  Consider the random dynamical system induced by the stochastic differential equation~\eqref{NormalForm}, and fix $a > 0$ and $\beta \in \mathbb{R}$. Then there exists a function $C: \R\times \R^+ \to \mathbb{R}^+$ such that if
  \[
  	b\geq C(\alpha, \sigma)\,,
  	\]
  then the top Lyapunov exponent $\lambda_{\rm top}$ is positive.
\end{theorem-nonD}
The random attractor $A$ is a random strange attractor in this situation, as illustrated in Figure~\ref{fig:chaosNF}~(e)--(h).

\subsection{Qualitative changes in the finite-time behaviour indicated by the dichotomy spectrum}

The final two main results concern the qualitative changes in the finite-time behaviour. If shear is small, then these changes occur at the deterministic Hopf bifurcation point $\alpha=0$, since the maximal finite-time Lyapunov exponents are equal to $\alpha$. If the shear is increased, then there is a transition to unbounded maximal finite-time Lyapunov exponents.

We also link these phenomena to qualitative changes in the dichotomy spectrum \cite{cdlr16}, which is based on the notion of an exponential dichotomy. We first need the concept of an invariant projector of a linear random dynamical system $(\theta:\R \times \Omega\to \Omega, \Psi:\R \times \Omega\to\R^{d\times d})$, which is given by a measurable function $P: \Omega \to \mathbb{R}^{d \times d}$ with
$$P(\omega) = P(\omega)^2 \ \text{ and } \ P(\theta_t  \omega) \Psi(t, \omega) = \Psi(t, \omega) P(\omega) \fa  t  \in \mathbb{R}  \text{ and } \omega \in \Omega\,.$$

\begin{definition}[Exponential dichotomy]
  Let $(\theta, \Psi)$ be a linear random dynamical system and let $\gamma \in \mathbb{R}$ and $P_{\gamma} : \Omega \to \mathbb{R}^{d \times d}$ be an invariant projector of $(\theta, \Psi)$. Then $(\theta, \Psi)$ is said to admit an \emph{exponential dichotomy} with \emph{growth rate} $\gamma \in \mathbb{R}$, constants $\alpha > 0$, $K \geq 1$ and projector $P_{\gamma}$ if for almost all $\omega \in \Omega$, one has
	\begin{align*}
	\| \Psi(t,\omega)P_{\gamma}(\omega)\| &\leq K e^{(\gamma - \alpha)t} \fa t \geq 0\,, \\
	\| \Psi(t,\omega)(\Id - P_{\gamma}(\omega))\| &\leq K e^{(\gamma + \alpha)t} \fa t \leq 0\,.
	\end{align*}
   We additionally define that $(\theta, \Psi)$ admits an exponential dichotomy with growth rate $\infty$ if there exists a $\gamma \in \mathbb{R}$ such that $(\theta, \Psi)$ admits an exponential dichotomy with growth rate $\gamma$ and projector $P_{\gamma} = \Id$. Analogously, $(\theta, \Psi)$ admits an exponential dichotomy with growth rate $-\infty$ if there exists a $\gamma \in \mathbb{R}$ such that $(\theta, \Psi)$ admits an exponential dichotomy with growth rate $\gamma$ and projector $P_{\gamma} = 0$.
\end{definition}

\begin{definition}[Dichotomy spectrum \cite{cdlr16}]
  Consider the linear random dynamical system $(\theta, \Psi)$. Then the dichotomy spectrum is defined by
  \begin{displaymath}
     \Sigma := \big\{ \gamma \in \mathbb{R}\cup\{-\infty,\infty\} : (\theta, \Psi) \text{ does not admit an exponential dichotomy with growth rate } \gamma\big\}\,.
  \end{displaymath}	
\end{definition}

Under the assumption of small shear, the following result describes a random bifurcation that corresponds to the deterministic Hopf bifurcation. The notions of uniform and finite-time attractivity are given precisely in Section~\ref{bifsmallsh}.

\begin{theorem-nonE}[Bifurcation for small shear]
  \label{smallshearresult}
  Consider the stochastic differential equation~(\ref{NormalForm}) with $\left|b \right| < a$. Then the random attractor $A$ is given by an attracting random equilibrium for all $\alpha \le 0$ and all $\alpha > 0$ in a neighbourhood of $0$. We observe the following bifurcation at $\alpha =0$:
  \begin{enumerate} [(i)]
  \item For $\alpha < 0$, the random equilibrium is globally uniformly attractive, but for $\alpha > 0$, the random equilibrium is not even locally uniformly attractive.
  \item Let $\Phi(t, \omega) := \rmD \varphi (t, \omega, A(\omega))$ denote the linearised random dynamical system along the random equilibrium for fixed $\alpha$. Then the dichotomy spectrum $\Sigma$ of $\Phi$ is given by
  \begin{equation*}
  \Sigma = [- \infty, \alpha] \,,
  \end{equation*}
  i.e.~hyperbolicity is lost at $\alpha = 0$.
  \item For $\alpha < 0$, the random equilibrium is finite-time attractive, whereas for $\alpha > 0$, it is not finite-time attractive.
  \end{enumerate}
\end{theorem-nonE}
The last result of the paper concerns the impact of shear on finite-time Lyapunov exponents. It implies a bifurcation of the spectrum of finite-time Lyapunov exponents for some critical value of shear $b^* \in [a,2a]$.
%

\begin{theorem-nonF}[Shear intensity as bifurcation parameter]
 \label{largeshearresult}
  Let $a,b, \sigma$ satisfy $b>2a>0$ and $\sigma\not=0$. Then for any $z\in\R^2$, the finite-time Lyapunov exponents of solutions starting in $z$ can be arbitrarily large and arbitrarily small with positive probability. More precisely, there exists a $T > 0$ such that for all $t \in (0,T]$, we have
  \[
  \esssup_{\omega\in\Omega}\sup_{\|v \|=1} \frac{1}{t}  \ln\|\rmD \varphi (t, \omega, z)v\|=\infty \quad \mbox{and}\quad \essinf_{\omega\in\Omega} \inf_{\|v \|=1} \frac{1}{t} \ln\|\rmD \varphi (t, \omega, z)v\|=-\infty.
  \]
\end{theorem-nonF}

\section{Generation of the random dynamical system and existence of a random attractor} \label{Generation}

We prove Theorem A in this section by following methods developed in \cite{imkellerlederer02,IS01}. We conjugate the SDE~\eqref{NormalForm} to a random differential equation via a suitable transformation using an Ornstein--Uhlenbeck process, so that we need to prove the existence of the random dynamical system and its random attractor for the corresponding random differential equation. An advantage in working with random differential equations (in comparison to \emph{stochastic} differential equations) is that we can work with sample path estimates of solutions.

For $c>0$, consider the stochastic differential equation
\begin{equation}\label{OU}
  \rmd Z = - c Z \rmd t + \rmd W_t\,,
\end{equation}
where $Z\in\R^2$. Define the random variable $Z^*:=\int_{-\infty}^0 e^{c s} \rmd W_s$. Then $t\mapsto Z^*(\theta_t\omega)$ solves \eqref{OU}, i.e.
\begin{equation}\label{OU_Eq1}
  Z^*(\theta_t\omega)
  =
  Z^*(\omega)
  -
  c\int_0^t Z^*(\theta_s\omega)\,\rmd s + \omega(t)\,.
\end{equation}
By replacing $\Omega$ with a measurable subset $\widehat\Omega\subset \Omega$ of full probability that is invariant under $\theta$, there exist two random variables $K$ and $L$ such that
	\begin{equation}\label{Bound_OU}
	|Z^*(\theta_t\omega)|^2
	\leq
	K(\omega)+ L(\omega)\ln(1+|t|) \fa t\in\mathbb R \mbox{ and } \omega \in\widehat\Omega\,,
	\end{equation}
see \cite{ks98}. We define the map $T:\widehat\Omega\times \mathbb R^2\rightarrow \mathbb R^2$ by $T(\omega,Z):=Z + \sigma Z^*(\omega)$.
Under the change of variable $Z\mapsto T(\omega,Z)$, the SDE~\eqref{NormalForm} is transformed into the random differential equation
\begin{equation}\label{RandomDifferentialEquation}
   \dot Z =g(\theta_t\omega,Z)\,,
\end{equation}
where $g(\omega,Z):=f\left(T(\omega,Z)\right)+c \sigma Z^*(\omega)$. We show later in Lemma~\ref{TechnicalLemma} and the proof of Theorem~A that the solution
$\Psi(t,\omega,Z)$ of this random differential equation,
\[
\Psi(t,\omega,Z)
=Z+ \int_0^t g(\theta_s\omega, \Psi(s,\omega,Z))\,\rmd s\,,
\]
exists for all $t\ge0$ and forms a random dynamical system. The following lemma holds using this fact.

\begin{lemma}\label{TransformSystems}
The following statements hold.
\begin{itemize}
  \item [(i)]
  The random dynamical system $\varphi: \mathbb R_{0}^+\times \widehat\Omega\times \mathbb R^2\rightarrow \mathbb R^2$, defined by
  \begin{equation}\label{Transformation_Evolution}
  \varphi(t,\omega,Z):=  \varphi(t,\omega,Z):= T(\theta_t\omega,\Psi(t,\omega,T(\omega)^{-1}Z))\,,
  \end{equation}
  is generated by the stochastic differential equation \eqref{NormalForm}.
  \item [(ii)]
  If the random dynamical system $\Psi$ has a random attractor, then also the random dynamical system $\varphi$ has a random attractor.
  \end{itemize}
\end{lemma}
\begin{proof}
(i)	From \eqref{Transformation_Evolution} and the definition of $T$, we have
\[
\varphi(t,\omega,Z)
=
\Psi(t,\omega,Z-\sigma Z^*(\omega))+\sigma Z^*(\theta_t\omega)\,,
\]
which together with the fact that $\Psi$ is a solution of \eqref{RandomDifferentialEquation} implies that
\begin{align*}
\varphi(t,\omega,Z)
&=
Z-\sigma Z^*(\omega)
+
\int_0^t g(\theta_s\omega,\Psi(s,\omega,Z-\sigma Z^*(\omega)))\,\rmd s + \sigma Z^*(\theta_t\omega)\\
&=
Z+\int_0^t g(\theta_s\omega, T(\theta_s\omega)^{-1}(\varphi(s,\omega,Z)))\,\rmd s+ \sigma (Z^*(\theta_t\omega)-Z^*(\omega))\,.
\end{align*}
Thus, using \eqref{OU_Eq1}, we obtain that
\begin{align*}
\varphi(t,\omega,Z)
&=
Z
+
\int_0^t g(\theta_s\omega, T(\theta_s\omega)^{-1}(\varphi(s,\omega,Z)))-\sigma c Z^*(\theta_s\omega)\,\rmd s + \sigma \omega(t)\\
&=
Z+ \int_0^t f(\theta_s\omega, \varphi(s,\omega,Z))\,\rmd s + \sigma \omega(t) \,,
\end{align*}
which completes the proof of this part.

\noindent
(ii) This follows from the definition of a random attractor and the fact that the shifted term in the transformation $T(\omega, Z)$, namely $Z^*(\theta_t\omega)$, is tempered.
\end{proof}

We show that the Euclidean norm of the solutions of \eqref{RandomDifferentialEquation} is bounded by the growth of the corresponding solutions of the scalar equation
\begin{equation}\label{Estimate_04}
\dot \zeta = \gamma_t(\omega)-\sqrt{a} \zeta\,,
\end{equation}
where the stochastic process $(\gamma_t)_{t\in\mathbb R}$ is chosen appropriately.
Note that for each initial value $\zeta_0\in\mathbb R$, the explicit solution of \eqref{Estimate_04} is given by
\begin{equation}\label{Estimate05}
\zeta(t,\omega,\zeta_0)
=
e^{-\sqrt{a}t}\zeta_0+\int_0^t e^{-\sqrt{a}(t-s)} \gamma_s(\omega)\,\rmd s\,.
\end{equation}

\begin{lemma}\label{TechnicalLemma}
  There exists a tempered stochastic processes $(\gamma_t)_{t\in\mathbb R}$, i.e.
  \begin{equation}\label{Temperness}
  \lim_{t \to \pm \infty}\frac{|\gamma_t(\omega)|}{e^{\epsilon |t|}}=0 \fa \epsilon>0 \mbox{ and }\omega\in\widehat\Omega\,,
  \end{equation}
  such that for $Z\in\mathbb R^2$, we have
  \begin{equation}\label{zetaest}
  \|\Psi(t,\omega,Z)\|^2
  \leq 2 \zeta(t,\omega,\|Z\|^2)\,,
  \end{equation}
  which implies that the solution $\Psi(t,\omega,Z)$ exists for all $t\ge0$.
\end{lemma}

\begin{proof}
  By replacing $Z$ with $(x,y)^\top$ and $Z^*$ with $(x^*,y^*)^\top$, we rewrite \eqref{RandomDifferentialEquation} as
  \begin{align*}
  \left(
  \begin{array}{l}
  \dot x_t\\
  \dot y_t
  \end{array}
  \right)
  &=
  \left(
  \begin{array}{ll}
  \alpha & -\beta\\
  \beta & \alpha
  \end{array}
  \right)
  \left(
  \begin{array}{l}
  x_t+\sigma x^*(\theta_t\omega) \\
  y_t+ \sigma y^*(\theta_t\omega)
  \end{array}
  \right)
  +
  c\sigma
  \left(
  \begin{array}{l}
   x^*(\theta_t\omega) \\
   y^*(\theta_t\omega)
  \end{array}
  \right)
  \\
  &
  -
  \left\|
  \left(
  \begin{array}{l}
  x_t+\sigma x^*(\theta_t\omega) \\
  y_t+ \sigma y^*(\theta_t\omega)
  \end{array}
  \right)
  \right\|^2
  \left(
  \begin{array}{ll}
  a & -b\\
  b & a
  \end{array}
  \right)
  \left(
  \begin{array}{l}
  x_t+\sigma x^*(\theta_t\omega) \\
  y_t+ \sigma y^*(\theta_t\omega)
  \end{array}
  \right)\,.
  \end{align*}
  Let $r_t:=\frac{1}{2}(x_t^2+y_t^2)$. Then a direct computation yields that
  \begin{align*}
  \dot r_t
  &= x_t\dot x_t+ y_t\dot y_t\\
  &=
  2 \alpha r_t + \sigma x^*(\theta_t\omega)((\alpha+c) x_t+\beta y_t)-\sigma y^*(\theta_t\omega) (\beta x_t-(\alpha+c) y_t)\\
  &\quad-\left\|
  \left(
  \begin{array}{l}
  	x_t+\sigma x^*(\theta_t\omega) \\
  	y_t+ \sigma y^*(\theta_t\omega)
  \end{array}
  \right)
  \right\|^2  \left(2 a r_t + \sigma x^*(\theta_t\omega)(a x_t+b y_t)-\sigma y^*(\theta_t\omega) (b x_t-a y_t)\right).
  \end{align*}
  Note that $\max\{(\alpha+c) x_t+\beta y_t, \beta x_t-(\alpha+c)y_t\}\leq \sqrt{((\alpha+c)^2+\beta^2)2r_t}$. Thus,
  \begin{eqnarray}
  &&|x^*(\theta_t\omega)((\alpha+c) x_t+\beta y_t)-y^*(\theta_t\omega) (\beta x_t-(\alpha+c) y_t)|\notag\\
  &&\leq
  \sqrt{((\alpha+c)^2+\beta^2)2r_t} \,(|x^*(\theta_t\omega)|+|y^*(\theta_t\omega)|)
  \notag
  \\
  &&\leq
  2\sqrt{((\alpha+c)^2+\beta^2)r_t}\, \|Z^*(\theta_t\omega)\| \,.\label{Esitmate_01}
  \end{eqnarray}
  On the other hand, we have
  \begin{displaymath}
  \left\|
  \left(
  \begin{array}{l}
  x_t+\sigma x^*(\theta_t\omega) \\
  y_t+ \sigma y^*(\theta_t\omega)
  \end{array}
  \right)
  \right\|^2
  =
  2 r_t+\sigma^2 \|Z^*(\theta_t\omega)\|^2 + 2\sigma x^*(\theta_t\omega) x_t+ 2\sigma y^*(\theta_t\omega)y_t\,,
  \end{displaymath}
  which together with the fact that
  $|x^*(\theta_t\omega)x_t+ y^*(\theta_t\omega)y_t|\leq \|Z^*(\theta_t\omega)\| \sqrt{2r_t}$ implies that
  \[
  \left|
  \left\|
  \left(
  \begin{array}{l}
  x_t+\sigma x^*(\theta_t\omega) \\
  y_t+ \sigma y^*(\theta_t\omega)
  \end{array}
  \right)
  \right\|^2
  -
  2r_t-\sigma^2\|Z^*(\theta_t\omega)\|^2
  \right|
  \leq
  2\sigma \|Z^*(\theta_t\omega)\|\sqrt{2r_t}\,.
  \]
  Consequently,
  \begin{equation}\label{Estimate_02}
  a r_t \left\|
  \left(
  \begin{array}{l}
  	x_t+\sigma x^*(\theta_t\omega) \\
  	y_t+ \sigma y^*(\theta_t\omega)
  \end{array}
  \right)
  \right\|^2
  \geq
  2ar_t^2 -2^{\frac{3}{2}}a\sigma \|Z^*(\theta_t\omega)\| r_t^{\frac{3}{2}} + a \sigma^2 	\|Z^*(\theta_t\omega)\|^2 r_t\,,
  \end{equation}
  and from the fact that
  \[
  |x^*(\theta_t\omega)(a x_t+b y_t)- y^*(\theta_t\omega) (b x_t-a y_t)|
  \leq
  \sqrt{a^2+b^2}\|Z^*(\theta_t\omega)\|\sqrt{2r_t}\,,
  \]
  we derive that
  \begin{equation}\label{Estimate_03}
  \begin{array}{l}
  \big|\sigma x^*(\theta_t\omega)(a x_t+b y_t)-\sigma y^*(\theta_t\omega) (b x_t-a y_t)\big|
  \left\|
  \left(
  \begin{array}{l}
  x_t+\sigma x^*(\theta_t\omega) \\
  y_t+ \sigma y^*(\theta_t\omega)
  \end{array}
  \right)
  \right\|^2\\
  \leq 2^{\frac{3}{2}}\sigma \sqrt{a^2+b^2} \|Z^*(\theta_t\omega)\| r_t^{\frac{3}{2}}+ \sigma^2 \sqrt{2(a^2+b^2)} \|Z^*(\theta_t\omega)\|^2 r_t \\
  + \sigma^3\sqrt{2(a^2+b^2)} \|Z^*(\theta_t\omega)\|^3 \sqrt{r_t}\,.
  \end{array}
  \end{equation}
  Using \eqref{Esitmate_01}, \eqref{Estimate_02}, \eqref{Estimate_03} and a comparison argument, we obtain  for all $t\ge0$  and $Z\in\mathbb R^2\setminus\{0\}$ that $\frac{1}{2}\|\Psi(t,\omega,Z)\|^2\leq \tilde\zeta(t,\omega,\|Z\|^2)$ , where $t\mapsto\tilde\zeta(t,\omega,\|Z\|^2)=\tilde\zeta_t$ is the solution of the following scalar differential equation
  \[
  \dot{\tilde{\zeta}}_t
  =
  a_t(\omega) \tilde\zeta_t^{\frac{1}{2}}
  +b_t(\omega) \tilde\zeta_t + c_t(\omega)\tilde\zeta_t^{\frac{3}{2}}-4 a \tilde\zeta_t^2\,,
  \]
  with initial condition $\tilde\zeta_0=\|Z\|^2$. Here the functions $a_t,b_t,c_t$ are defined by
  \begin{align*}
  a_t(\omega)
  &:=
  2 \sigma \sqrt{(\alpha+c)^2+\beta^2}\|Z^*(\theta_t\omega)\|+ \sqrt{2} \sigma^3 \sqrt{a^2+b^2}\|Z^*(\theta_t\omega)\|^2\,,\\
  b_t(\omega)
  &:=
  2 \alpha+ 4\sigma^2\sqrt{a^2+b^2} \|Z^*(\theta_t\omega)\|^2 - 2 a\sigma^2 \|Z^*(\theta_t\omega)\|^2\,,\\
  c_t(\omega)
  &:=
  \left(2^{3/2}\sqrt{a^2+b^2}\sigma+ 2^{5/2} a \sigma \right) \|Z^*(\theta_t\omega)\|\,.
  \end{align*}
  From temperdness of $Z^*(\theta_t\omega)$, all stochastic processes $(a_t)_{t\in\mathbb R}, (b_t)_{t\in\mathbb R}$ and $(c_t)_{t\in\mathbb R}$	are also tempered. Note that
  \begin{align*}
  	a \tilde\zeta_t^2+ \sqrt[3]{\frac{a_t(\omega)^4}{4^4 a}}
  	&\geq |a_t(\omega)|\tilde\zeta_t^{\frac{1}{2}}\,,\\
  	a \tilde\zeta_t^2+ \frac{b_t(\omega)^2}{2a}
  		&\geq |b_t(\omega)\tilde\zeta_t|\,,\\
  	a \tilde\zeta_t^2+ 	\frac{3^3 c_t(\omega)^4}{4^4 a^3}
  	&\geq |c_t(\omega)\tilde\zeta_t^{\frac{3}{2}}|\,.
  \end{align*}
  Therefore,
  \begin{displaymath}
  a_t(\omega) \tilde\zeta_t^{\frac{1}{2}}
  +b_t(\omega) \tilde\zeta_t + c_t(\omega)\tilde\zeta_t^{\frac{3}{2}}-4 a \tilde\zeta_t^2
  \le
  \sqrt[3]{\frac{a_t(\omega)^4}{4^4 a}}
  +
  \frac{b_t(\omega)^2}{2a}
  +
  \frac{3^3 c_t(\omega)^4}{4^4 a^3}-a \tilde\zeta_t^2\\
  \leq
  \gamma_t(\omega)-\sqrt{a}\tilde\zeta_t\,,
  \end{displaymath}
  where
  \[
  \gamma_t(\omega):=\frac{1}{4}+\sqrt[3]{\frac{a_t(\omega)^4}{4^4 a}}+\frac{b_t(\omega)^2}{2a} +\frac{3^3 c_t(\omega)^4}{4^4 a^3}
  \]
  is tempered. Hence, using a comparison argument, the solution $\zeta$ of \eqref{Estimate_04} satisfies \eqref{zetaest}, which finishes the proof of this lemma.
\end{proof}

\begin{proof}[Proof of Theorem A]
  (i) According to \cite{a98}, there exists a local random dynamical system generated by solutions of \eqref{RandomDifferentialEquation}. Due to Lemma~\ref{TechnicalLemma}, the solution $\Psi(t,\omega,Z)$ exists for all $t\ge0$. Hence, this proves the fact that we assumed to prove Lemma~\ref{TransformSystems}. Lemma~\ref{TransformSystems}~(i) completes the proof of (i).

  \noindent
  (ii)  Let $D\in \mathcal F\otimes \mathcal B(\R^d)$  be tempered. Then there exists a tempered random variable $R:\widehat\Omega\rightarrow \mathbb R^+$ such that $D(\omega)\subset B_{R(\omega)}(0)$. By Lemma~\ref{TechnicalLemma}, for all $Z\in D(\theta_{-t}\omega)$, we have
  \begin{displaymath}
  \|\Psi(t,\theta_{-t}\omega,Z)\|^2
   \leq 2 \zeta(t,\theta_{-t}\omega, R(\theta_{-t}\omega))\\
  \leq
  2 e^{-\sqrt{a}t} R(\theta_{-t}\omega) + 2 \int_{-t}^0 e^{\sqrt{a}s} \gamma_s(\omega)\,\rmd s\,,
  \end{displaymath}
  where we use \eqref{Estimate05} to obtain the last inequality. Since $(\gamma_t)_{t\in\mathbb R}$ is tempered,
  $\int_{-\infty}^0 e^{\sqrt as} \gamma_s(\omega)\,\rmd s$ exists. On the other hand, since $R$ is tempered, it follows that $\lim_{t \to \infty} e^{-\sqrt a t} R(\theta_{-t}\omega)=0$. Define $r(\omega):=\sqrt{1+2 \int_{-\infty}^0 e^{\sqrt{a}s} \gamma_s(\omega)\,\rmd s}$. Thus, for each $\omega\in\widehat{\Omega}$, there exists $T>0$ such that
  \[
  \Psi(t,\theta_{-t}\omega,D(\theta_{-t}\omega))
  \subset
  B_{r(\omega)}(0) \fa  t\geq T\,.
  \]
  This means that $B_{r(\omega)}(0)$ is an absorbing set. Applying Theorem~\ref{ExistencePullback} completes the proof.
\end{proof}

\section{Synchronisation}\label{Hopf}

We prove in this section that the system~\eqref{NormalForm} admits synchronisation if the top Lyapunov exponent is negative (Theorem~B), and we show that small implies negativity of the top Lyapunov exponent and thus synchronisation (Theorem~C). In addition, we show that the system satisfies the integrability condition of the Multiplicative Ergodic Theorem, and we prove that the sum of the two Lyapunov exponents is always negative.

\subsection{Negativity of the sum of the Lyapunov exponents}
Recall that $\Phi:\R_0^+\times \Omega\times \R^2\to \R^{2\times 2}$ is the linear random dynamical system satisfying $\Phi(0,\omega,Z)=\id$ and
\[
\dot \Phi(t,\omega,Z)=
\rmD f(\varphi(t,\omega,Z)))\Phi(t,\omega,Z)\,.
\]
We show that $\Phi$ satisfies the integrability condition of the Multiplicative Ergodic Theorem with respect to the measure $\mu$ and also show that the sum of the Lyapunov exponents of $\Phi$ is always negative.
\begin{proposition}\label{Singularity}
  The following statements hold.
  \begin{itemize}
    \item [(i)]
    Let $\lambda^+:\mathbb R^2\rightarrow \mathbb R$ be defined by
    \begin{equation}\label{UpperFunction}
      \lambda^+(Z):=\max_{\|r\|=1} \langle Df(Z)r,r\rangle\,.
    \end{equation}
    Then for almost all $\omega\in \Omega$ and all $Z\in \R^2$, we have
    \begin{equation}\label{UpperEstimate}
    \|\Phi(t,\omega,Z)\|
    \leq
    \exp\left(\int_0^t \lambda^+(\varphi(s,\omega,Z))\,\rmd s\right) \fa t\geq 0\,,
    \end{equation}
    and the linear random dynamical system $\Phi$ satisfies the integrability condition of the Multiplicative Ergodic Theorem.
    \item [(ii)] Let $\lambda_{\Sigma}$ be the sum of the two Lyapunov exponents of the linear random dynamical system $\Phi$.
    Then $\lambda_{\Sigma}<0$ and the disintegrations of the Markov measure $\mu$ are singular with respect to the Lebesgue measure on $\mathbb R^2$.
  \end{itemize}
\end{proposition}
\begin{proof}
  (i) Let $v \in \mathbb{R}^2\setminus\{0\}$ be arbitrary. By definition of $\Phi$, we have
  \begin{align*}
  	\frac{\rmd}{\rmd t} \| \Phi(t, \omega, Z)v \|^2 &= 2 \big\langle \rmD f( \varphi(t,\omega,Z)) \Phi(t, \omega, Z) v, \Phi(t, \omega, Z)v \big\rangle \\
  	& = 2 \left\langle \rmD f( \varphi(t,\omega,Z))\frac{ \Phi(t, \omega, Z) v}{\| \Phi(t, \omega, Z) v\|} , \frac{ \Phi(t, \omega, Z) v}{\| \Phi(t, \omega, Z) v\|} \right\rangle \| \Phi(t, \omega, Z)v \|^2 \\[1.5ex]
  	& \leq 2 \lambda^+( \varphi(t,\omega,Z)) \| \Phi(t, \omega, Z)v \|^2.
  \end{align*}
  This implies that
  \begin{equation} \label{NotDodgy}
  	\| \Phi(t, \omega, Z)v \|^2 \leq \| v \|^2 \exp \left(2 \int_0^t \lambda^+( \varphi(s,\omega,Z)) \rmd s \right)\, .
  \end{equation}
  Since $v$ is arbitrary, \eqref{UpperEstimate} is proved. Using \eqref{UpperEstimate}, we obtain that
  \[
  \sup_{0 \leq t \leq 1} \ln^+ \| \Phi(t, \omega,Z) \|
  \leq \int_0^1 |\lambda^+( \varphi(s,\omega,Z))|\, \rmd s\,,
  \]
  which implies that
  \begin{align}
  \int_{\Omega \times \mathbb{R}^2} \sup_{0 \leq t \leq 1} \ln^+ \| \Phi(t, \omega,Z) \| \ \rmd \mu(\omega,Z)
  &\leq
  \int_{\Omega \times \mathbb{R}^2}\int_0^1|\lambda^+( \varphi(s,\omega,Z))|\; \rmd s\,\rmd  \mu(\omega,Z)\notag\nonumber\\
  &=
  \int_0^1\int_{\Omega \times \mathbb{R}^2}|\lambda^+( \varphi(s,\omega,Z))|\; \rmd  \mu(\omega,Z) \,\rmd s\notag\nonumber\\
  &=
  \int_{\mathbb R^2} |\lambda^{+}(Z)|\, \rmd  \rho(Z)\,,\label{UpperEstimate2}
  \end{align}
  where in the last equality, we use the fact that the skew product $\Theta_s(\omega,Z)=(\theta_s\omega,\varphi(s,\omega,Z))$ preserves the probability measure $\mu$. By definition of $\lambda^+$ and the explicit form of $\rmD f$ given by
  \[
  \rmD f(Z)
  =
  \left(
  \begin{array}{ll}
  	\alpha & -\beta\\
  	\beta & \alpha
  \end{array}
  \right)
  -
  \left(
  \begin{array}{ll}
  	3ax^2 +2bxy+ay^2 & b x^2+2a xy+3by^2\\
  	-3bx^2-2axy-by^2 & ax^2-2bxy+3ay^2
  \end{array}
  \right)\,,
  \]
  it follows that
  \[
  |\lambda^+(Z)|
  \leq
  |\alpha|+6(|a|+|b|)(x^2+y^2) \fa  Z=(x,y)^{\top}\in \mathbb R^2\,.
  \]
  Together with \eqref{UpperEstimate2}, this implies that
  \[
  \int_{\Omega \times \mathbb{R}^2} \sup_{0 \leq t \leq 1} \ln^+ \| \Phi(t, \omega,Z) \| \, \rmd \mu(\omega,Z)
  \leq
  |\alpha|+ 6(|a|+|b|)\int_{\mathbb R^2} (x^2+y^2) p(x,y)\,\rmd  x\,\rmd y\,,
  \]
  where $p(x,y)$ is given as in \eqref{StatDens}. Thus, the linear random dynamical system $\Phi$ satisfies the integrability condition of the Multiplicative Ergodic Theorem.

  (ii) Due to $\lambda_{\Sigma} =  \lim_{t \to \infty} \frac{1}{t} \ln \det \Phi(t,\omega,Z)$, the sum of the two Lyapunov exponents of the linear random dynamical system generated by \eqref{Linearization} reads as
  	\[
  	\lambda_{\Sigma}
  	=
  	2\alpha
  	-4a\int_{\R^2}(x^2+y^2) p(x,y)\,\rmd x \,\rmd y\,.
  	\]
  	Using the explicit formula for $p(x,y)$ from \eqref{StatDens}, we obtain that
  	\[
  	\lambda_{\Sigma}
  	=
  	2\alpha
  	-4a
  	\frac{\int_{\R^2} (x^2+y^2)\exp\left(\frac{2\alpha (x^2+y^2)- a(x^2+y^2)^2}{2\sigma^2}\right)\, \rmd x\, \rmd y}{\int_{\R^2}\exp\left(\frac{2\alpha (x^2+y^2)- a(x^2+y^2)^2}{2\sigma^2}\right)\, \rmd x\,\rmd y}\,.
  	\]
  	Applying the change of variables $x= \sigma r \sin\phi$, $y=\sigma r\cos\phi$ the previous integral yields that
  	\[
  	\lambda_{\Sigma}
  	=
  	2\alpha
  	-4a\sigma^2
  	\frac{\int_0^{\infty }r^3 \exp\left(\frac{2\alpha r^2 -a\sigma^2 r^4}{2}\right)\,\rmd r}{\int_0^{\infty }r \exp\left(\frac{2\alpha r^2 -a\sigma^2 r^4}{2}\right)\,\rmd r}\,.
  	\]
  	A further change of variable $r^2\mapsto r$ gives that
  	\begin{displaymath}
  		\lambda_{\Sigma}
  		=
  		2\alpha
  		-4a\sigma^2
  		\frac{\int_0^{\infty }r \exp\left(\frac{2\alpha r -a\sigma^2 r^2}{2}\right)\,\rmd r}{\int_0^{\infty } \exp\left(\frac{2\alpha r -a\sigma^2 r^2}{2}\right)\,\rmd r}\,,
  	\end{displaymath}
  which proves that $\lambda_{\Sigma}<0$ if $\alpha\leq 0$. We also show this for $\alpha>0$ now. Using the change of variable $\sqrt a |\sigma| r -\frac{\alpha}{\sqrt a |\sigma|} \mapsto r$, we obtain that
  \begin{align*}	
  \lambda_{\Sigma}		&=
  		-2\alpha
  		-4\sqrt{a}|\sigma|
  		\frac{\int_{-\frac{\alpha}{|\sigma|\sqrt a}}^{\infty} r \exp\left(-\frac{r^2}{2}\right)\,\rmd r }{\int_{-\frac{\alpha}{|\sigma|\sqrt a}}^{\infty} \exp\left(-\frac{r^2}{2}\right)\,\rmd r}\\
  		&=
  		-2\alpha - 4\sqrt{a}|\sigma| \frac{\exp\left( -\frac{\alpha^2}{2a\sigma^2}\right)}{\int_{-\frac{\alpha}{|\sigma|\sqrt a}}^{\infty} \exp\left(-\frac{r^2}{2}\right)\,\rmd r}\,,
  	\end{align*}
  which shows that $\lambda_{\Sigma}<0$ for $\alpha >0$. As a consequence, using \cite[Proposition~1]{lj87} and \cite[Theorem~4.15]{b91}, the disintegration of the Markov measure $\mu$ is singular with respect to $\rho$ if $\lambda_{\Sigma}<0$. The fact that $\rho$ is equivalent to the Lebesgue measure finishes the proof of this proposition.
\end{proof}	

\subsection{Negative top Lyapunov exponent implies synchronisation}

The aim of this subsection is to prove synchronisation of the random dynamical system generated by \eqref{NormalForm} when its top Lyapunov exponent is negative. Our proof consists of two ingredients. The first ingredient is a result from \cite{fgs16} that implies that the fibers of the random attractor are singletons. The second ingredient is the stable manifolds theorem, which we use to verify that this random attractor is also attractive in forward time.

We make use of the following sufficient conditions for the collapse of a random attractor \cite[Theorem~2.14]{fgs16}.

\begin{theorem}[Collapse of the random attractor] \label{SynchroTheo}
  We assume that a random dynamical system $(\theta, \varphi)$  is
  \begin{enumerate}
  \item[(i)] \emph{asymptotically stable} on a fixed non-empty open set $U\subset \R^2$, in the sense that there exists a sequence $t_n \to \infty$ such that
  \begin{equation*}
  \mathbb{P}\left( \omega\in\Omega: \lim_{n \to \infty} \diam(\varphi(t_n,\omega,U)) =0 \right)> 0\,.
  \end{equation*}
  \item[(ii)] \emph{swift transitive}, i.e.~for all $x,y\in\R^2$ and $r>0$, there exists a $t>0$ such that
  \begin{equation*}
  \mathbb{P}\big( \omega\in\Omega: \varphi(t,\omega,B_r(x)) \subset B_{2r}(y) \big) > 0\,.
  \end{equation*}
  \item[(iii)] \emph{contracting on large sets}, i.e.~for all $R >0$, there exist $y\in\R^2$ and $t>0$ such that
  \begin{equation*}
  \mathbb{P} \left( \omega\in\Omega:\diam(\varphi(t,\omega,B_R(y))) \leq \tfrac{R}{4} \right) > 0\,.
  \end{equation*}
  \end{enumerate}
  Suppose further that $(\theta,\varphi)$ has a random attractor $A$ with $\mathcal{F}_{-\infty}^0$-measurable fibers. Then $A(\omega)$ is a singleton $\mathbb{P}$-almost surely.
\end{theorem}

We use this result for the following proposition.

\begin{proposition}\label{Pullback_Attractor}
  Suppose that the top Lyapunov exponent $\lambda_{\rm top}$ of the random dynamical system generated by \eqref{NormalForm} is negative. Then the fibers of the random attractor are singletons, given by $\mathcal{F}_{-\infty}^{0}$-measurable map $A:\Omega\rightarrow \mathbb R^2$. Furthermore, the following statements hold:
  \begin{itemize}
  	\item [(i)] $A$ is a random equilibrium of $\varphi$, i.e.
  	\[
  	\varphi(t,\omega,A(\omega))=A(\theta_t\omega)\fa t\geq 0 \mbox{ and almost all } \omega\in\Omega\,.
  	\]
  	\item [(ii)] The random equilibrium is distributed according to the stationary density $(x,y)\mapsto p(x,y)$, see~\eqref{StatDens}. More precisely,
  	\[
  	\mathbb P\big(\{\omega \in \Omega: A(\omega)\in C\}\big)
  	=\int_{C} p(x,y) \, \rmd x \,\rmd y \fa  C \in\mathcal B(\mathbb R^2)\,.
  	\]
  	\item [(iii)] The top Lyapunov exponent of the linearization along the random equilibrium $a$,
  	\[
  	\dot \xi = \rmD f(A(\theta_t\omega))\xi\,,
  	\]
  is equal to $\lambda_{\rm top}$.
  \end{itemize}
\end{proposition}
\begin{proof}
  In the first part of the proof, we show that the random dynamical system $\varphi$ generated by \eqref{NormalForm} fulfils the assumptions (i), (ii), and (iii) of Theorem~\ref{SynchroTheo}. Note that (i) follows from the negativity of the top Lyapunov exponent (see \cite[Lemma 4.1 and Corollary 4.4]{fgs16}), and swift transitivity holds for our system according to \cite[Proposition 4.9]{fgs16}. Hence, it remains to show contraction on large sets for $\varphi$. By definition of $f$, we have that
  \begin{align*}
   \big\langle f(x) - f(y),  x-y\big\rangle  \leq &\left(\alpha - a\tfrac{1}{2}(\|x\|^2 + \|y\|^2) \right) \| x- y \|^2 \\
  & + b (x_1 y_2 - y_1 x_2) \left(2\langle x-y, y\rangle + \|x - y\|^2 \right).
  \end{align*}
  Fix $r >0$, and consider $B_r(z)$, where $z = (R,0)$ for some $R>0$ to be chosen large enough. For any $x, y \in B_r(z)$, observe that
  \begin{equation*}
    (x_1 y_2 - y_1 x_2)\langle x-y, y\rangle  \leq r \|y\|  \| x- y \|^2 +  r^2  \| x- y \|^2
  \end{equation*}
  and
  \begin{equation*}
   (x_1 y_2 - y_1 x_2) \|x - y\|^2 \leq 2 \|x - y\|^2 \|y\|  \| x- y \|\,.
  \end{equation*}
  This implies that for all $x, y \in B_r(z)$,
  \begin{align*}
  \langle f(x) - f(y), x-y\rangle &\leq \| x- y \|^2 \left(\alpha - a \tfrac{1}{2}(\|x\|^2 + \|y\|^2) + 2b ( r \|y\|  +  r^2 + 2 \|y\| r)  \right)\\
   &< K \| x- y \|^2
  \end{align*}
  for some $K < 0$ if $R$ is big enough (due to the quadratic terms, $K$ has negative sign). This property is called monotonicity on large sets, which implies contraction on large sets due to~\cite[Proposition~3.10]{fgs16}.

  We now prove the statements (i), (ii) and (iii) of the proposition.

  (i) This follows immediately from the definition of a random attractor (see Appendix).

  (ii)  Note that $\omega \mapsto A(\omega)$ is measurable with respect to $\mathcal{F}_{- \infty}^0$, and thus, $\mu_{\omega} := \delta_{A(\omega)}$ defines a Markov measure. The invariance of $\mu_{\omega}$ follows directly from (i). Hence, $\{\mu_\omega\}_{\omega\in\Omega}$ is the disintegration of the ergodic invariant measure $\mu$ associated with the ergodic stationary measure $\rho$, and we obtain from \eqref{Expectedmeasure} that for all $C \in \mathcal{B}(\mathbb{R}^2)$
  $$ \mathbb P(\{\omega\in \Omega: A(\omega)\in C\}) = \int_{\Omega} \delta_{A(\omega)}(C) \; \rmd \mathbb{P}(\omega)  = \int_{\Omega} \mu_{\omega}(C) \; \rmd \mathbb{P}(\omega) = \rho(U) =\int_{C} p(x,y) \; \rmd x \,\rmd y \,.$$

  (iii) According to the Multiplicative Ergodic Theorem, the existence of the Lyapunov spectrum holds for a set $M\subset \Omega \times \mathbb{R}^2$ of full $\mu$-measure. We observe that the set
  $$D = \bigcup_{\omega \in \Omega} \{(\omega, A(\omega))\}$$ has full $\mu$-measure, since
  $$ \mu(D) = \int_{\Omega} \mu_{\omega} (\{A(\omega)\}) \; \rmd \mathbb{P}(\omega) = \int_{\Omega} \delta_{A(\omega)} (\{A(\omega)\}) \; \rmd \mathbb{P}(\omega) = 1\,. $$
  Hence, $\mu ( M \cap D) =1 $. Since the Oseledets space associated with the second Lyapunov exponent has zero Lebegue measure for any $(\omega, x) \in M \cap D$, the claim follows.
\end{proof}

Finally, we prove Theorem B.

\begin{proof}[Proof of Theorem  B]
The existence of the attracting random equilibrium $A:\Omega\rightarrow \mathbb R^2$ has been shown in Proposition~\ref{Pullback_Attractor}.
Define $\psi:\mathbb R_{0}^+\times \Omega\times \mathbb R^2\rightarrow \mathbb R^2$ by
\[
\psi(t,\omega,x)=\varphi(t,\omega,A(\omega)+x)-\varphi(t,\omega,A(\omega)).
\]
Obviously, $\psi(t,\omega,0)=0$ and $\psi(t,\omega,x)$ is the solution of the random differential equation
\begin{equation}\label{PerturbedSystems}
\dot \xi= \rmD f(A(\theta_t\omega))\xi + R(t,\omega,\xi),
\end{equation}
where
\[
R(t,\omega,\xi):=
f(A(\theta_t\omega)+\xi)-f(A(\theta_t\omega))-Df(A(\theta_t\omega))\xi\,.
\]
Note that for $R\equiv 0$, the top Lyapunov exponent of the homogeneous equation \eqref{PerturbedSystems} is negative. Using the stable manifold theorem \cite[Theorems~7.5.5 and 7.5.16]{a98}, there exists $r(\omega)>0$ such that for almost all $\omega\in\Omega$ and $x\in B_{r(\omega)}(0)$, one has
\begin{equation}\label{Stablemanifold_Application}
\lim_{t \to \infty}
e^{-\frac{\lambda_{\rm top}}{2} t} \|\psi(t,\omega,x)\|=\lim_{t \to \infty}
e^{-\frac{\lambda_{\rm top}}{2} t}\|\varphi(t,\omega,x+A(\omega))-A(\theta_t\omega)\|=0\,.
\end{equation}
Choose and fix an arbitrary initial value $x\in \mathbb R^2$, and define
\[
V:=\left\{\omega\in\Omega: \limsup_{t \to \infty} e^{-\frac{\lambda_{\rm top}}{2}t} \|\varphi(t,\omega,x)-A(\theta_t\omega)\|=0\right\}.
\]
It remains to show that $\mathbb P(V)=1$. For each $n\in\mathbb N$, we define
\[
\Omega_n:=\left\{
\omega\in\Omega:
\varphi(t,\theta_{-t}\omega,x)\in B_{r(\omega)}(A(\omega)) \mbox{ for all } t\geq n\right\}.
\]
Note that $(\Omega_n)_{n\in\mathbb N}$ is an increasing sequence of measurable sets. By virtue of Proposition \ref{Pullback_Attractor}, the random equilibrium $a$ is the random attractor of $\varphi$, which implies $\lim_{n \to \infty}\mathbb P(\Omega_n)=1$. From the definition of $\Omega_n$, we derive that  $\varphi(n,\theta_{-n}\omega,x)\in B_{r(\omega)}(A(\omega))$ for all $\omega\in\Omega_n$. Together with \eqref{Stablemanifold_Application}, this implies that for all $\omega\in\Omega_n$, one has
\begin{align*}
0&=\limsup_{t \to \infty} e^{-\frac{\lambda_{\rm top}}{2}t} \|\varphi(t,\omega,\varphi(n,\theta_{-n}\omega,x))-A(\theta_t\omega)\|\\[0.5ex]
&=
\limsup_{t \to \infty} e^{-\frac{\lambda_{\rm top}}{2}t} \|\varphi(t+n,\theta_{-n}\omega,x)-A(\theta_t\omega)\|\,.
\end{align*}
 Consequently, $\theta_{-n}\Omega_n\subset V$, and thus, $\mathbb P(V)=1$, which finishes the proof.
\end{proof}

\subsection{Small shear implies synchronisation}

We prove Theorem~C in this subsection, which says that small shear implies negativity of the top Lyapunov exponent. The main ingredient for the proof of Theorem~C is the inequality in Proposition~\ref{Singularity}(i).

We first need the following estimate on the function $\lambda^+$ defined as in \eqref{UpperFunction}.
\begin{lemma}\label{Inequality_Df}
For any $Z=(x,y)^\top\in \mathbb R^2$, we have
\[
\lambda^+(Z)\leq  \alpha + \big(\sqrt{a^2+b^2}-2a\big)(x^2 +y^2)\,,
\]
and equality holds if and only if $xy=0$.
\end{lemma}
\begin{proof}
  Using the following explicit form of $\rmD f(Z)$,
  \begin{equation*}
    Df(Z) = \begin{pmatrix}
    \alpha - a y^2 - 3a x^2 -2byx & -\beta -2axy -bx^2-3by^2\\
    \beta- 2axy + by^2 + 3bx^2&  \alpha - ax^2 - 3ay^2 + 2byx
    \end{pmatrix},
  \end{equation*}
  we obtain for any $ r \in \mathbb{R}^2$ with $\|r\|=1$ that
  \begin{align*}
      \langle Df(x,y)r,r\rangle  &= r_1^2(\alpha - ay^2 - 3ax^2) + r_1 r_2 (-\beta - 2axy) + r_1 r_2 (\beta - 2axy) + r_2^2 (\alpha - ax^2 - 3ay^2) \\
  	&\quad -2byx r_1^2 + 2byx r_2^2 + r_1 r_2(2bx^2-2by^2)\\
  	&= \alpha - a(x^2 +y^2)+ 2b(r_1r_2x^2-r_1r_2y^2+(r_2^2-r_1^2)xy) - 2 a(r_1 x + r_2 y)^2\,.
  \end{align*}
  Since $r_1^2+r_2^2=1$, it is possible to write that $r_1=\sin\phi$ and $r_2=\cos \phi$ for some $\phi\in [0,2\pi)$. Thus, a simple calculation yields that
  \begin{align*}
  \langle Df(x,y)r,r\rangle
  &=
  \alpha -2a (x^2+y^2)+ (ax^2+bxy-ay^2)\cos2\phi+ (bx^2-2axy-by^2)\sin 2\phi\\
  &\leq
  \alpha -2a (x^2+y^2)+\sqrt{(ax^2+2bxy-ay^2)^2+(bx^2-2axy-by^2)^2}\\
  &=
  \alpha -2a (x^2+y^2)+\sqrt{(a^2+b^2)(x^2-y^2)^2+ 4b^2 x^2y^2}\\
  &\leq
  \alpha -2a (x^2+y^2)+\sqrt{a^2+b^2} (x^2+y^2)\,,
  \end{align*}
  which completes the proof.
\end{proof}

\begin{proof}[Proof of Theorem C]
  From inequality \eqref{UpperEstimate}, we derive that 	
  	\begin{equation*}
  	\lambda_{\mathrm{top}} =\limsup_{t\to\infty}\frac{1}{t}\ln\|\Phi(t,\omega,s)\|
  \leq \lim_{t \to \infty} \frac{1}{t} \int_0^t \lambda^+( \varphi(s,\omega,x))\, \rmd s\,.
  	\end{equation*}
  Note that the skew product flow $\Theta_s(\omega,Z)=(\theta_s\omega,\varphi(s,\omega,Z))$ preserves the probability measure $\mu$, and $\lambda^{+}$ is integrable. By using Birkhoff's Ergodic Theorem, we obtain that
  \[
  \lambda_{\mathrm{top}}
  \leq
  \int_{\mathbb{R}^2} \lambda^+(x,y) p(x,y)\, \rmd x \,\rmd y\,,
  \]
  where the density function $p$ is as in \eqref{StatDens}. Thus, by virtue of Lemma \ref{Inequality_Df}, we arrive at
  \begin{equation*}\label{Inequality_LE}
    \lambda_{\mathrm{top}}
    <
    \alpha + \big(\sqrt{a^2+b^2}-2a\big) \int_{\mathbb{R}^2}(x^2+y^2)  p(x,y)\, \rmd x \,\rmd y\,.
  \end{equation*}
  Inserting the explicit form of the density function $p$ in the preceding inequality  gives that
  \begin{equation}\label{LambdaTop_Ineq}
  \lambda_{\mathrm{top}}
  <
  \alpha+ (\sqrt{a^2+b^2}-2a)K \int_{\mathbb{R}^2}(x^2+y^2) \exp \left( \frac{2 \alpha (x^2 + y^2) - a(x^2 + y^2)^2}{2 \sigma^2} \right) \,\rmd x \,\rmd y\,,
  \end{equation}
  with the normalization constant $K=\frac{ 2 \sqrt{2 a} }{\sqrt{\pi} \sigma \erfc \big(- \alpha/\sqrt{2 a \sigma^2}   \big)}$.
  Using polar coordinates, we obtain that
  	\begin{align*}
  	&\quad K \int_{\mathbb{R}^2} \left(\alpha - a(x^2 + y^2) \exp \left( \frac{2 \alpha (x^2 + y^2) - a(x^2 + y^2)^2}{2 \sigma^2} \right)\right)\,\rmd x \,\rmd y \\
  	&= 2 \pi K \int_{0}^{\infty} (\alpha - a r^2) r \exp \left( \frac{2 \alpha r^2 - ar^4}{2 \sigma^2} \right)\, \rmd r \\
  	&= - \pi \sigma^2 K\,.
  	\end{align*}
  This implies
  \begin{equation*}
    \int_{\mathbb{R}^2} (x^2 + y^2) \exp \left( \frac{2 \alpha (x^2 + y^2) - a(x^2 + y^2)^2}{2 \sigma^2} \right)\,\rmd x \,\rmd y
    =
    \frac{\alpha}{Ka}+\frac{\pi\sigma^2}{a}\,,
  \end{equation*}
  which together with \eqref{LambdaTop_Ineq} implies that
  \begin{equation*}
  \lambda_{\mathrm{top}}
  <
  \alpha+\big(\sqrt{a^2+b^2}-2a\big)\left(
  \frac{\alpha}{a}+\frac{\pi K \sigma^2}{a}
  \right)\,.
  \end{equation*}
  Consequently,
  \begin{equation}\label{Integral}
  \lambda_{\mathrm{top}}
  <
  -\pi K\sigma^2
  +
  \left(\sqrt{1+\frac{b^2}{a^2}}-1\right)(\alpha+\pi K \sigma^2)\,.
  \end{equation}
  Note that by definition of $K$ it is easy to see that $\alpha+\pi K \sigma^2> 0$. Therefore, for all $|b|\leq \kappa$, we have
  \[
  \lambda_{\mathrm{top}}
  <
  -\pi K\sigma^2+\left(\sqrt{1+\frac{\kappa^2}{a^2}}-1\right)(\alpha+\pi K \sigma^2)=0\,,
  \]	
  which completes the proof of this theorem.
\end{proof}

\section{Random Hopf bifurcation} \label{Random Hopf}

We analyse random bifurcations for the stochastic differential equation~\eqref{NormalForm} in this section, which captures qualitative changes in the the asymptotic as well as the finite-time behaviour.

We first need the following preparatory proposition.

\begin{proposition} \label{smallneighbour}
  Consider \eqref{NormalForm} such that $|b|\leq \kappa$. Then for any $y \in \mathbb{R}^2$, $\epsilon > 0$ and $T \geq 0$, there exists a set $E \in \mathcal{F}_{-\infty}^T$ with $\mathbb{P}(E)>0$ such that
  \begin{displaymath}
     A(\theta_s \omega) \in B_{\epsilon}(y) \fa s \in [0,T] \mbox{ and }  \omega \in E\,,
  \end{displaymath}
  where $\{A(\omega)\}$  is the unique random equilibrium for \eqref{NormalForm} from Proposition \ref{Pullback_Attractor}.
\end{proposition}

\begin{proof}
  Let $\epsilon > 0$ and $T \geq 0$. Since $\Omega = \bigcup_{x \in \mathbb{Q}^2} \{ \omega \in\Omega\,:\, A(\omega) \in B_{\epsilon/4}(x) \}$,
  there exists an $x \in \mathbb{R}^2$ such that $$A_0 := \{ \omega \in\Omega\,:\, A(\omega) \in B_{\epsilon/4}(x)\}$$ has positive measure. From \cite[Proposition 3.10]{fgs16} we know that there exists $t_0 > 0$ such that
  $$
  B_0 := \left\{ \omega\in\Omega\,:\, \varphi (t_0,\omega, x') \in B_{\epsilon/2}(y) \mbox{ for all }x' \in B_{\epsilon/4}(x) \right\}$$
  has positive measure.
  Since $\theta$ is measure preserving, the two sets
  \begin{align*}
   A_1 &:= \theta_{t_0} A_0 = \{ \omega\in\Omega\,:\, A(\theta_{-t_0} \omega) \in B_{\epsilon/4}(x)\}\,, \\
   B_1  &:= \theta_{t_0} B_0 = \left\{ \omega\in\Omega\,:\, \varphi(t_0,\theta_{-t_0} \omega, x') \in B_{\epsilon/2}(y) \mbox{ for all } x' \in B_{\epsilon/4}(x) \right\}
  \end{align*}
  have positive measure.
  Due to the Markov property of the random dynamical system, we observe that $B_1$ and $A_1$ are independent, and hence, $\mathbb{P}(B_1 \cap A_1) > 0$. Thus, the set
  $$ E_0 = \{ \omega \in\Omega\,:\, A(\omega) \in B_{\epsilon/2}(y)\} \supset A_1 \cap B_1 $$
  has positive measure and clearly lies in $ \mathcal{F}_{-\infty}^{0}$. Fix $\omega \in E_0$.
  Similarly to the proof of \cite[Proposition~3.10]{fgs16}, define
  $$ h(t) := - \frac{t f(A(\omega))}{\sigma} \fa  t \in[0,T]\,,$$
  where $f$ denotes the vector field of the drift in \eqref{MainEq}. We write $\varphi(t,g,z)$, $t\in[0,T]$, for the solution of \eqref{MainEq} with initial condition $z$ and path $g \in C_0^T := \{ \bar g \in C([0,T], \mathbb{R}^2)\,:\, \bar g(0)=0 \}$.
  We can infer that $\varphi(t,h,A(\omega)) =  A(\omega)$ for all $t \in [0,T]$. Recall that the map $g \mapsto \varphi(\cdot,g,z)$ is continuous from $C_0^T$ to $C([0,T], \mathbb{R}^2)$ with respect to the supremum norm $\| \cdot \|_{\infty}$. Hence, there is a $\delta > 0$ such that for all $g \in C_{\delta} := \{ \bar g \in C_0^T\,:\, \| \bar g- h \| \leq \delta \}$, we have
  $$ \| \varphi(t,g, A(\omega)) - \varphi(t,h, A(\omega)) \| < \epsilon/2  \fa t \in [0,T]\,.$$
  Since the set $E_+ := \big\{\omega\,:\, \omega|_{[0,T]} \in C_{\delta} \big\}$ has positive measure and is independent of $E_0$, the set  $E = E_0 \cap E_+ \in \mathcal{F}_{-\infty}^{T}$ has positive measure and satisfies $$ A(\theta_t \omega) \in B_{\epsilon}(y) \fa t \in [0,T] \mbox{ and }\omega \in E\,,$$
  by the above construction.
\end{proof}

\subsection{Bifurcation for small shear} \label{bifsmallsh}

In this subsection, we consider the stochastic differential equation~\eqref{NormalForm} with small enough shear such there exists a random equilibrium for $\alpha$ close to zero. We prove in Theorem~\ref{Uniattract} that the random equilibrium $A:\Omega\to \R^2$ loses uniform attractivity at the deterministic bifurcation point $\alpha = 0$. On the other hand, we will observe a loss of hyperbolicity at the bifurcation point in the dichotomy spectrum associated with the random equilibrium. Moreover, we can show that $A:\Omega\to \R^2$ is finite-time attractive before, but not after the bifurcation point, indicated by a transition from zero to positive probability of positive finite-time Lyapunov exponents.

We call the random attractor $A$ \emph{locally uniformly attractive} if there exists a $\delta > 0$ such that
\begin{displaymath}
  \lim_{t \to \infty} \textstyle \sup_{x \in B_{\delta}(0)} \operatorname{ess} \operatorname{sup}_{\omega \in \Omega} \| \varphi(t, \omega ,A(\omega) + x) - A(\theta_t \omega) \| = 0\,.
\end{displaymath}
We call it \emph{globally uniformly attractive} if the above holds for any $\delta > 0$.

\begin{theorem}\label{Uniattract}
  Consider the stochastic differential equation \eqref{NormalForm} such that there is a unique attracting random equilibrium $A:\Omega\to \R^2$ (see Proposition~\ref{Pullback_Attractor} and Theorem C). Then for $\alpha < 0$ and $\left| b \right| \leq a$, the random attractor $A:\Omega\to\R^2$ is globally uniformly attractive. Furthermore, for all pairs of initial conditions $U,V\in \mathbb R^2$, we have
  \[
  \|\phi(t,\omega, U)-\phi(t,\omega,V)\|
  \leq e^{\alpha t} \|U-V\|\fa t\ge 0\,.
  \]
  For $\alpha > 0$, the random attractor $A:\Omega\to\R^2$ is not even locally uniformly attractive.
\end{theorem}

\begin{proof}
  Fix $\alpha < 0$, and choose arbitrary $U,V\in\mathbb R^2$, $\omega\in\Omega$. Define
  \[
  \begin{pmatrix}
  x_t\\
  y_t
  \end{pmatrix}
  :=\phi(t,\omega,U)
  \quad\mbox{ and } \quad
  \begin{pmatrix}
  \widehat x_t\\
  \widehat y_t
  \end{pmatrix}
  :=\phi(t,\omega,V)\,.
  \]
  From \eqref{NormalForm}, we derive that
  \begin{align*}
  \frac{\rmd}{\rmd t}
  \begin{pmatrix}
  x_t-\widehat x_t\\
  y_t-\widehat y_t
  \end{pmatrix}
  =&
  \begin{pmatrix}
  \alpha & -\beta\\
  \beta & \alpha
  \end{pmatrix}
  \begin{pmatrix}
  x_t-\widehat x_t\\
  y_t-\widehat y_t
  \end{pmatrix}
  -\\
  &
  \quad(x_t^2+y_t^2)
  \begin{pmatrix}
  a & b\\
  -b & a
  \end{pmatrix}
  \begin{pmatrix}
  x_t\\
  y_t
  \end{pmatrix}
  +(\widehat x_t^2+\widehat y_t^2)
  \begin{pmatrix}
  a & b\\
  -b & a
  \end{pmatrix}
  \begin{pmatrix}
  \widehat x_t\\
  \widehat y_t
  \end{pmatrix}\,.
  \end{align*}
  Therefore,
  \begin{align*}
  \frac{1}{2}\frac{\rmd}{\rmd t}\left\|
  \begin{pmatrix}
  x_t-\widehat x_t\\
  y_t-\widehat y_t
  \end{pmatrix}
  \right\|^2
  &=
  (x_t-\widehat x_t)\frac{\rmd}{\rmd t}(x_t-\widehat x_t)+(y_t-\widehat y_t)\frac{\rmd}{\rmd t}(y_t-\widehat y_t)\\
  &=
  \alpha \left\|
  \begin{pmatrix}
  x_t-\widehat x_t\\
  y_t-\widehat y_t
  \end{pmatrix}
    \right\|^2
  - R(x_t,\widehat x_t,y_t,\widehat y_t)\,,
  \end{align*}
  where
  \[
  R(x_t,y_t,\widehat x_t,\widehat y_t):=
  a\left(r_t^2+\widehat r_t^2- (x_t\widehat x_t+y_t\widehat y_t)(r_t+\widehat r_t)\right)
  +b (x_t\widehat y_t-\widehat x_t y_t)(r_t-\widehat r_t)
  \]
  with $r_t:=x_t^2+y_t^2$ and $\widehat r_t:=\widehat x_t^2+\widehat y_t^2$. To show global uniform attractivity, it is sufficient to establish that $R(x_t,y_t,\widehat x_t,\widehat y_t)\geq 0$. From the inequality $(|xy|+|uv|)^2\leq (x^2+u^2)(y^2+v^2)$, we derive that
  \begin{align*}
   &
  |(x_t\widehat x_t+y_t\widehat y_t)(r_t+\widehat r_t)|+| (x_t\widehat y_t-\widehat x_t y_t)(r_t-\widehat r_t)|\\
  &\leq
  \sqrt{(x_t\widehat x_t+y_t\widehat y_t)^2+  (x_t\widehat y_t-\widehat x_t y_t)^2}
  \sqrt{(r_t+\widehat r_t)^2+ (r_t-\widehat r_t)^2}\\
  &=
  \sqrt{2r_t\widehat r_t (r_t^2+\widehat r_t^2)}\\
  &\leq
  r_t^2+\widehat r_t^2\,.
  \end{align*}
  Together with the fact that $|b|\leq a$, this implies that $R(x_t,y_t,\widehat x_t,\widehat y_t)\geq 0$, which establishes global uniform attractivity for $\alpha<0$.

  We assume now that $\alpha > 0$. Suppose to the contrary that there exists $\delta > 0$ such that
  $$ \lim_{t \to \infty} \textstyle\sup_{x \in B_{\delta}(0)} {\ess \sup}_{\omega \in \Omega} \| \varphi(t, \omega, A(\omega) + x) - A(\theta_t \omega) \| = 0\,.$$
  This implies that there exists an $N \in \mathbb{N}$ such that for all $t > N$, we have
  $$\textstyle\sup_{x \in B_{\delta}(0)} {\ess \sup}_{\omega \in \Omega} \| \varphi(t, \omega, A(\omega) + x) - A(\theta_t \omega) \| < \frac{1}{4}\sqrt{\frac{\alpha}{a}}\,.$$
  Due to Proposition~\ref{smallneighbour}, there exists a positive measure set $E_0 \in \mathcal{F}_{-\infty}^0$ such that $A(\omega) \in B_{\delta/4}(0)$ for all $\omega\in E_0$. Let $\phi(\cdot,x_0)$ denote the solution of the deterministic equation \eqref{NormalForm} for $\sigma=0$ with initial condition $x(0) = x_0$. Then there exists a $T > N$ such that
  \begin{displaymath}
    \textstyle\| \phi(T, (\pm \frac{1}{4}\delta,0)) \| > \frac{1}{2}\sqrt{\frac{\alpha}{a}}\,,
  \end{displaymath}
  and at the same time
  \begin{displaymath}
    \textstyle\| \phi(T, (\frac{1}{4}\delta,0)) - \phi(T, (- \frac{1}{4}\delta,0)) \| > \sqrt{\frac{\alpha}{a}} \,.
  \end{displaymath}
  Recall from the proof of Proposition~\ref{smallneighbour} that $\omega \mapsto \varphi(\cdot,\omega,x)$ is continuous from $C_0^T$ to $C([0,T], \mathbb{R}^2)$ with respect to the supremum norm. This implies that there exists an $\epsilon > 0$ such that for all $\omega \in E_{\epsilon} = \{ \omega  \in \Omega \,:\, \sup_{t \in [0,T]} \| \omega(t) \| < \epsilon \} \in \mathcal{F}_{0}^T$, we obtain
  \begin{displaymath}
    \textstyle\phi(T, (\frac{1}{4}\delta,0)) - \varphi(T,\omega, (\frac{1}{4}\delta,0)) \| < \frac{1}{4}\sqrt{\frac{\alpha}{a}} \quad \mbox{ and } \quad \| \phi(T, (-\frac{1}{4}\delta,0)) - \varphi(T, \omega,(- \frac{1}{4}\delta,0)) \| < \frac{1}{4}\sqrt{\frac{\alpha}{a}}\,.
  \end{displaymath}
  This implies that
  \begin{displaymath}
    \textstyle\| \varphi(T, \omega, (\frac{1}{4}\delta,0)) - \varphi(T, \omega, (- \frac{1}{4}\delta,0)) \| > \frac{1}{2}\sqrt{\frac{\alpha}{a}} \,.
  \end{displaymath}
  Since $E_{\epsilon}$ and $E_{0}$ are independent sets of positive measure, we get that $\mathbb{P}(E) > 0$ where $E = E_{\epsilon} \cap E_{0}$. However, for all $\omega \in E$, we conclude
  \begin{align*}
  &\quad\sup_{x \in B_{\delta}(0)} \| \varphi(t, \omega,A(\omega) + x) - A(\theta_t \omega) \| \\
  &\geq \textstyle\max \left\{ \big\| \varphi(t, \omega, (\frac{1}{4}\delta,0)) - A(\theta_t \omega) \big\|\,, \ \big\| \varphi(t, \omega,(-\frac{1}{4}\delta,0)) - A(\theta_t \omega) \big\| \right\} > \frac{1}{4}\sqrt{\frac{\alpha}{a}}\,,
  \end{align*}
  which contradicts our assumption.
\end{proof}

We show now that this loss of uniform attractivity at the deterministic bifurcation point is associated with a change of sign in the dichotomy spectrum.

\begin{theorem} \label{Dichotomysmallshear}
  Consider the stochastic differential equation \eqref{NormalForm} such that there exists a unique attracting random equilibrium $A:\Omega\to\R^2$ (see Proposition~\ref{Pullback_Attractor} and Theorem C). Let $\Phi(t, \omega) := \rmD \varphi (t, \omega, A(\omega))$ denote the linearized random system along the random equilibrium. Then for $\left| b \right| < a$ and $\alpha \in \mathbb{R}$ small enough such that the random equilibrium $A:\Omega\to\R^2$ exists, the dichotomy spectrum $\Sigma$ of $\Phi$ is given by
  \begin{equation*}
  \Sigma = [- \infty, \alpha] \,.
  \end{equation*}
\end{theorem}

\begin{proof}
  Recall from Proposition~\ref{Singularity} that we have
  $$\| \rmD\varphi(t, \omega, x)\| \leq  \exp \left( \int_0^t \lambda^+( \varphi(s,\omega,x)) \rmd s \right)\,.$$
  Since Lemma~\ref{Inequality_Df} implies that $\lambda^+ (x) \leq \alpha - (a - \left|b \right|) \|x\|^2$, we have
  \begin{equation} \label{fromabove}
  \| \Phi(t, \omega) \| \leq \exp \left( \int_0^t  \big(\alpha - (a - \left|b \right|) \|A(\theta_s \omega))\|^2 \big)\,\rmd s \right)\,.
  \end{equation}
  Similarly, with $\lambda^-(x):=\min_{\|r\|=1} \langle Df(x)r,r\rangle$, we have
  $$ \| \rmD\varphi(t, \omega, x)\| \geq  \exp \left( \int_0^t \lambda^-( \varphi(s,\omega,x)) \rmd s \right)\,.$$
  It is easy to see that $\lambda^- (x) \geq \alpha - 4 a \|x\|^2$, which implies
  \begin{equation} \label{frombelow}
  \| \Phi(t, \omega)  \| \geq \exp \left( \int_0^t (\alpha - 4 a \| A(\theta_s \omega))\|^2) \rmd s \right)\,.
  \end{equation}
  From \eqref{fromabove} we can deduce immediately that for almost all $\omega \in \Omega$, we have
  \begin{equation*}
  \| \Phi(t, \omega) \| \leq e^{\alpha \left| t \right|}  \fa t \in \mathbb{R}\,.
  \end{equation*}
  This implies that $\Sigma \subset (-\infty, \alpha]$.

  We now show that $(-\infty, \alpha] \subset \Sigma$. Choose $\gamma \in (-\infty, \alpha]$, and suppose to the contrary that $\Phi$ admits an exponential dichotomy with growth rate $\gamma$ with an invariant projector $P_{\gamma}$ and constants $K, \epsilon > 0$.
We consider the following three cases (note that the rank of the invariant projector does not depend on $\omega$, see \cite{cdlr16}):
  \begin{enumerate}
    \item[(i)]
      $P_{\gamma} \equiv \id$. This means that for almost all $\omega \in \Omega$,
      \begin{displaymath}
        \|\Phi(t, \omega) \|\leq K e^{(\gamma - \epsilon)t} \fa t \geq 0 \,.
      \end{displaymath}
      Fix $T > 0$ such that $e^{\frac{1}{5}\epsilon T} > K$. According to Proposition~\ref{smallneighbour}, there exists a positive measure set $E$ such that for all $\omega \in E$ and $s \in [0,T]$, we have $A(\theta_s \omega) \in B_ {\sqrt{\epsilon/(5a)}}(0)$. We derive from \eqref{frombelow} that for such $\omega\in E$, we have
      \begin{displaymath}
        \| \Phi(T, \omega) \| \geq e^{T(\alpha - \frac{4 }{5}\epsilon )} \geq K e^{(\gamma - \epsilon)T} \,.
      \end{displaymath}
      This contradicts the assumption.
    \item[(ii)]
    $\rank P_{\gamma} \equiv 1$. The argument is the same as in the previous case, since our estimates do not depend on the tangent vector $v$, but hold for the norm $\|\Phi(t, \omega) \|$.
    \item[(iii)]
    $P_{\gamma} \equiv 0$. This means that for almost all $\omega \in \Omega$, we have
  $$ \| \Phi(t, \omega) \|  \geq \frac{1}{K} e^{(\gamma + \epsilon)t} \fa t \geq 0 \,.$$
  Together with \eqref{fromabove}, this implies that
  $$ \frac{\ln K + (\alpha - \epsilon - \gamma)t}{a - \left| b \right|} \geq \int_0^t \|A(\theta_s \omega)\|^2 \rmd s\,. $$
  Choose some $T > 1$ and $y \in \mathbb{R}^2$ such that
  $$ \|y\|^2 > 4 \max \left\{ \frac{\ln K }{a - \left| b \right|}, \frac{\alpha - \epsilon - \gamma}{a - \left| b \right|}  \right\}\,.$$
  Take $\delta < \frac{\|y\|}{2}$. Then by Proposition~\ref{smallneighbour}, there exists a set $E\in \mathcal{F}_{- \infty}^T$ such that
  $$ A(\theta_s \omega) \in B_{\delta}(y) \fa s \in [0,T] \mbox{ and } \omega \in E\,.$$
  This implies
  $$ \int_0^t \|A(\theta_s \omega)\|^2 \rmd s > T \frac{\|y\|^2}{4} > \frac{\ln K + (\alpha - \epsilon - \gamma)T}{a - \left| b \right|} \,, $$
  which is a contradiction.
  \end{enumerate}
  This finishes the proof of this theorem.
\end{proof}

We demonstrate now that the change of sign in the dichotomy spectrum is mirrored by finite-time properties of the system. To see this, consider a compact time interval $[0,T]$ and the corresponding \emph{finite-time top Lyapunov exponents} associated with the attractive random equilibrium $A:\Omega\to \R^2$, given by
\begin{displaymath}
  \lambda^{T, \omega} := \sup_{\|v\|=1} \frac{1}{T} \ln \| \Phi(T, \omega)v \| \fa \omega\in\Omega\,.
\end{displaymath}
From Proposition~\ref{Pullback_Attractor}~(iii), we obviously have  $\lambda_{\mathrm{top}} = \lim_{T \to \infty} \lambda^{T, \omega}$ almost surely, where $\lambda_{\mathrm{top}}$ is the top Lyapunov exponent of \eqref{NormalForm}.
\begin{proposition} \label{Finitetime}
Consider the stochastic differential equation \eqref{NormalForm} with $\left|b \right| < a$ and $\alpha \in \mathbb{R}$ such that there exists a unique attractive random equilibrium $A:\Omega\to\R^2$. The following statements hold.
\begin{enumerate}[(i)]
\item For $ \alpha < 0$, we have $\lambda^{T, \omega} \leq \alpha < 0$ for all $\omega  \in \Omega$, which means that the random attractor $A:\Omega\to\R^2$ is \emph{finite-time attractive}.
\item For $ \alpha > 0$, we have $\mathbb{P} \big( \omega \in \Omega \,:\, \lambda^{T, \omega} > 0 \big) > 0$, which means that the random attractor $A:\Omega\to\R^2$ is not finite-time attractive.
\end{enumerate}
\end{proposition}

\begin{proof}
(i) Recall from \eqref{fromabove} that
$$
\| \Phi(t, \omega) \| \leq \exp \left( \int_0^t \big( \alpha - (a - \left|b \right|) \|A(\theta_s \omega))\|^2 \big)\,\rmd s \right)\,,
$$
which implies that
\begin{equation*}
\lambda^{T, \omega} \leq \frac{1}{T} \int_0^T \big( \alpha - (a - \left|b \right|) \|A(\theta_s \omega))\|^2 \big)\,\rmd s \leq \alpha < 0\,.
\end{equation*}

(ii) Recall from \eqref{frombelow} that
\begin{displaymath}
  \| \Phi(t, \omega)  \| \geq \exp \left( \int_0^t \big(\alpha - 4 a \| A(\theta_s \omega))\|^2 \big)\,\rmd s \right)\,.
\end{displaymath}
Choose $\epsilon := \sqrt{\frac{\alpha}{5 a}} > 0$. According to Proposition~\ref{smallneighbour}, there exists a set $E \in \mathcal{F}_{-\infty}^T$ of positive measure such that $A(\theta_s \omega) \in B_{\epsilon}(0)$ for all $s \in [0,T]$ and $\omega\in E$. Then
\begin{equation*}
\lambda^{T, \omega} \geq \frac{1}{T} \int_0^T \big(\alpha - 4 a \| A(\theta_s \omega))\|^2 \big)\,\rmd s \geq \alpha - \frac{4 \alpha}{5} = \frac{ \alpha}{5} > 0 \fa \omega\in E\,.
\end{equation*}
This shows the claim.
\end{proof}

\begin{proof}[Proof of Theorem E]
  The claims follow from Theorem~\ref{Uniattract}, Theorem~\ref{Dichotomysmallshear} and Proposition~\ref{Finitetime}.
\end{proof}
The proofs of Theorem~\ref{Dichotomysmallshear} and Proposition~\ref{Finitetime} explain in detail how the change of finite-time attractivity is connected to the loss of hyperbolicity in the dichotomy spectrum. Due to \cite[Theorem~4.5]{cdlr16}, we obtain
\begin{equation}\label{rel1}
  \lim_{T \to \infty} \esssup_{ \omega \in \Omega} \lambda^{T, \omega} = \lim_{T \to \infty} \esssup_{ \omega \in \Omega} \sup_{\|v\|=1} \frac{1}{T} \ln \| \Phi(T, \omega)v \| = \sup \Sigma\,.
\end{equation}
A similar statement holds for the infimum of the dichotomy spectrum. This means that the finite-time Lyapunov exponents are, at least asymptotically, supported on the dichotomy spectrum, and having positive values in the spectrum implies that, at least asymptotically, we can observe positive finite-time Lyapunov exponents.

\subsection{Shear intensity as bifurcation parameter}

We now do not assume the existence of an attractive random equilibrium, and we aim at proving Theorem F in this subsection. We first show a statement that corresponds to Proposition~\ref{smallneighbour} in this more general context.

\begin{proposition} \label{smallneighbour2}
Let $(\theta, \varphi)$ be the random dynamical system generated by \eqref{NormalForm}, and let $x, y \in \mathbb{R}^2$, $\epsilon > 0$ and $T > 0$. Then for any $t_0 \in (0, T]$, there exists a set $E \in \mathcal{F}$ with $\mathbb{P}(E)>0$ such that
\begin{displaymath}
  \varphi(s, \omega,x) \in B_{\epsilon}(y) \fa s \in [t_0,T] \mbox{ and } \omega \in E\,.
\end{displaymath}
\end{proposition}

\begin{proof}
Similarly as in the proof of \cite[Proposition 3.10]{fgs16}, fix $t_0 \in (0, T]$ and define
$$ \psi(t) := x + \frac{t}{t_0}(y-x) \fa t \in[0,t_0]\,,$$
and
$$ h(t) := \frac{1}{\sigma} \left( \psi(t) - x - \int_0^t f(\psi(s))\, \rmd s \right) \fa  \ t \in[0,T]\,,$$
where $f$ denotes the vector field of the drift in \eqref{NormalForm}. As in the proof of Proposition~\ref{smallneighbour}, we write $\varphi(t,g,z)$ for the solution of \eqref{NormalForm} with initial condition $z$ and path $g\in C_0^T$.
We can infer that $\varphi(t,h,x) = \psi(t)$ for all $ t \in [0,t_0]$, and in particular $\varphi(t_0,h,x) = y$. Recall that the map $g \mapsto \varphi(\cdot,g,z)$ is continuous from $C_0^T$ to $C([0,t_0], \mathbb{R}^2)$ with respect to the supremum norm $\| \cdot \|_{\infty}$. This implies that there exists a $\delta > 0$ such that for all $g \in C_{\delta} := \big\{ b \in C_0^T : \| b- h \| \leq \delta \big\}$, we have
$$ \| \varphi(t,g, x) - \varphi(t,h, x) \| < \tfrac{1}{2}\epsilon  \fa   t \in [0,t_0]\,.$$
Hence, we have established that there is a positive measure set $E_1 := \big\{\omega\in\Omega: \omega|_{[0,t_0]} \in C_{\delta} \big\}$ such that for all $\omega \in E_1$, we have
$\varphi(t_0, \omega,x) \in B_{\epsilon/2} (y)$.

Similar to this argument, one can construct a set $E_2$ of positive measure that is independent from $E_1$ (by the Markov property) such that for all $\omega \in E := E_1 \cap E_2$, we have
$$ \varphi(t, \omega,x) \in B_{\epsilon} (y) \fa t \in [t_0, T]\,.$$
This finishes the proof of this proposition.
\end{proof}

\begin{proof}[Proof of Theorem F]
  For $(\omega,z) \in \Omega \times \mathbb{R}^2$ and $\alpha \in \mathbb{R}$, the linear random dynamical system $t\mapsto \Phi(t, \omega,z)$ is solution of the variational equation
	\begin{equation*}
	\frac{\rmd}{\rmd t} \Phi(t, \omega, z) = \rmD f( \varphi(t,\omega,z)) \Phi(t, \omega, z), \quad \mbox{where } \Phi(0,\omega,z) = \Id\,.
	\end{equation*}
  Define $s_t(\omega,z,v) := \frac{\Phi(t, \omega, z) v}{\| \Phi(t, \omega, z) v\|}$ and observe that for $v\in\R^2\setminus\{0\}$,
	\begin{align*}
	\frac{\rmd}{\rmd t} \| \Phi(t, \omega, z)v \|^2 &= 2 \left\langle \rmD f( \varphi(t,\omega,z)) \Phi(t, \omega, z) v, \Phi(t, \omega, z)v \right\rangle \\
	& = 2 \left\langle \rmD f( \varphi_t(\omega,z))s_t(\omega,z,v) , s_t(\omega,z,v) \right\rangle \| \Phi(t, \omega, z)v \|^2\,.
	\end{align*}
Let $\mu >0$, and let $z' = (w,w) \in \mathbb{R}^2$ be such that $ \frac{b-2a}{2} \|z'\|^2 = (b-2a) w^2 \geq \mu$ and $w >1$. Note that
	\begin{equation*}
	\rmD f(x,y) = \begin{pmatrix}
	\alpha - a y^2 - 3a x^2 -2byx & -\beta -2axy -bx^2-3by^2\\
	\beta- 2axy + by^2 + 3bx^2&  \alpha - ax^2 - 3ay^2 + 2byx
	\end{pmatrix}\,.
	\end{equation*}
With $\tilde r= (0,1)$, we get $$\langle Df(z')\tilde r,\tilde r\rangle = \alpha + 2(b -2 a)w^2 \geq \alpha + 2\mu \,.  $$
Let
$$\epsilon = \min \left\{ 1, \frac{1}{16} \frac{b - 2a}{b w}, \frac{\sqrt{b -2a}}{4 a} \right\} \quad \mbox{and} \quad \delta = \frac{1}{8} \frac{b - 2a}{4b} \,.$$
Then by Proposition~\ref{smallneighbour2}, there is a positive measure set $E_1\subset \Omega$ such that for all $\omega \in E_1$
$$ \varphi(t, \omega,z') \in B_{\epsilon}(z') \fa t \in [0,1]\,.$$
This implies that the coefficients of $\rmD f( \varphi(t,\omega,z'))$ are bounded uniformly in $\omega\in E_1$ for $t\in [0,1]$. Because $\Phi$ is continuous, there is a $T  \in (0,1]$ such that
$$ \| s_t(\omega,z',\tilde r) - \tilde r \| = \left \| \frac{\Phi(t, \omega, z') \tilde r}{\| \Phi(t, \omega, z') \tilde r\|} - \tilde r \right\| <  \delta \fa t \in [0,T] \mbox{ and }\omega \in E_1$$
Note that we obtain for any $r \in \mathbb{R}^2$ with $\|r\|=1$
\begin{align*}
\langle Df(x,y)r,r\rangle &= r_1^2(\alpha - ay^2 - 3ax^2) + r_1 r_2 (-\beta - 2axy) + r_1 r_2 (\beta - 2axy) + r_2^2 (\alpha - ax^2 - 3ay^2) \\
&\quad -2byx r_1^2 + 2byx r_2^2 + r_1 r_2(2bx^2-2by^2)\\
&= \alpha - a(x^2 +y^2)+ 2b(r_1r_2x^2-r_1r_2y^2+yx(r_2^2-r_1^2)) - 2 a(r_1 x + r_2 y)^2 \,.
\end{align*}
This means that for all $t \in (0,T]$ and $\omega \in E_1$, we have by the choice of $\epsilon$ and $\delta$ above that
\begin{align*}
&\quad \left \langle \rmD f( \varphi_t(\omega,z))s_t(\omega,z,\tilde r) , s_t(\omega,z,\tilde r) \right\rangle \\
&\geq \alpha - 2a(w + \epsilon)^2 + 2b(w - \epsilon)^2(1 - 2 \delta) - 2b \delta [ (w + \epsilon)^2 - (w - \epsilon)^2] - 2 a (w + \epsilon)^2\\
& = \alpha +  (b-2a)w^2 + \big((b-2a)w^2  - 4a (2 w \epsilon + \epsilon^2) - 4 b w \epsilon - \delta 4b(w - \epsilon)^2 - 2 \delta 4 b w \epsilon\big) \\&\geq \alpha + \mu\,.
\end{align*}
Hence, we get that for all $\omega \in E_1$ and $t \in (0,T]$, the finite-time top Lyapunov exponent of trajectories starting in $z'$ satisfies
$$\lambda^{t,\omega, z'} := \sup_{\|v\| =1} \frac{1}{t} \ln \| \Phi(t, \omega, z')v\|\geq \alpha + \mu\,.$$
Since $\mu > 0$ was arbitrary, we obtain with positive probability arbitrarily large finite-time Lyapunov exponents when starting in $z'$.

We now show that for any $z\in\R^2$ and $t_0\in (0,T]$, the finite-time top Lyapunov exponent $\lambda^{t,\omega, z}$, $t\in[t_0,T]$, can be arbitrarily large for $\omega$ from a set of positive measure.
By Proposition~\ref{smallneighbour2}, there exists a set $E_2 \in \mathcal{F}$ with $\mathbb{P}(E_2)>0$ such that
$$ \varphi(s, \omega,z) \in B_{\epsilon}(z') \fa s \in [t_0,T] \mbox{ and } \omega \in E_2\,,$$
where the values of $\varphi(t,\omega,z)$, $t\in [0,t_0]$, stay close to the line between $z$ and $z'$ (see proof of Proposition~\ref{smallneighbour2}).
Since $t_0$ can be chosen arbitrarily small and the solutions stay in a compact set for $t\in[0,t_0]$, we obtain with similar arguments as before that with positive probability there are arbitrarily large finite-time Lyapunov exponents.

Let $\mu^- < 0$. Then by choosing $z''=(w,-w)$, we obtain with similar arguments as above that for some $T \in (0,1]$
$$\inf_{\|v\| =1} \frac{1}{t} \ln \| \Phi(t, \omega, z'')v\| \leq \alpha + \mu^- \fa t \in [0,T] \mbox{ and } \omega \mbox{ from a set of positive probability}\,.$$
By using Proposition~\ref{smallneighbour2} again, we can then deduce that with positive probability, there are arbitrarily small finite-time Lyapunov exponents for any initial conditions.
\end{proof}

\section*{Acknowledgments} The authors would like to thank Alexis Arnaudon, Darryl Holm, Nikolas N{\"u}sken, Grigorios Pavliotis and Sebastian Wieczorek for useful discussions. Maximilian Engel was supported by a Roth Scholarship from the Department of Mathematics at Imperial College London. Jeroen S.W.~Lamb acknowledges the support by Nizhny Novgorod University through the grant RNF 14-41-00044, and Martin Rasmussen was supported by an EPSRC Career Acceleration Fellowship EP/I004165/1. This research has also been supported by EU Marie-Curie IRSES Brazilian-European Partnership in Dynamical Systems (FP7-PEOPLE-2012-IRSES 318999 BREUDS) and EU Marie-Sk\l odowska-Curie ITN Critical Transitions in Complex Systems (H2020-MSCA-2014-ITN 643073 CRITICS).


\bibliographystyle{plain}

\bibliography{mybibfile}

\appendix

\section*{Appendix}

\section{Lyapunov spectrum}
\label{LyapSpec}


A random dynamical system $(\theta,\varphi)$ is called \emph{linear} if the map $\varphi(t,\omega):\mathbb R^d\to\mathbb R^d$, $x\mapsto \varphi(t,\omega,x)$, is linear for any $(t,\omega)\in \R\times \Omega$. Define $\Phi:\R\times \Omega\to \R^{d\times d}$ by $\Phi(t,\omega)x:= \varphi(t,\omega,x)$. Suppose that $\Phi$ satisfies the integrability condition
\[
  \sup_{0 \leq t \leq 1} \ln^+ \| \Phi(t, \omega) \| \in L^1(\mathbb P)\,,
\]
where $\ln^+(x):= \max\{\ln(x),0\}$. Then the Multiplicative Ergodic Theorem \cite[Theorem 3.4.1, Theorem 4.2.6]{a98} guarantees the existence of a $\theta$-forward invariant set $\widehat \Omega \subset \Omega$ with $\mathbb P (\widehat \Omega) = 1$, the Lyapunov exponents $\lambda_1 > \dots > \lambda_p$, and an invariant measurable filtration
\begin{displaymath}
\mathbb R^d = V_1(\omega) \supsetneq V_2(\omega)\supsetneq \dots \supsetneq V_p(\omega) \supsetneq V_{p+1}(\omega)= \{0\}\,,
\end{displaymath}
such that  for all $0 \neq x \in \mathbb{R}^d$, the \emph{Lyapunov exponent} $\lambda(\omega, x)$, defined by
\begin{equation*}
  \lambda(\omega, x) = \lim_{t \to \infty} \frac{1}{t} \ln \| \Phi(t, \omega) x \|
\end{equation*}
exists, and we have
\begin{equation*}
  \lambda (\omega, x) = \lambda_{i}\quad \Longleftrightarrow \quad x \in V_i(\omega) \setminus V_{i+1}(\omega) \fa i\in\{1, \dots, p\}\,.
\end{equation*}
\section{Random attractors}\label{sec-randomattr}

A random variable $R:\Omega\rightarrow \mathbb{R}$ is called \emph{tempered} if
\[
\lim_{t \to \pm \infty} \frac{1}{|t|} \ln^{+} R(\theta_t\omega)=0 \faa \omega\in \Omega\,,
\]
see also \cite[p.~164]{a98}. A set $D\in \mathcal F\otimes \mathcal B(\R^d)$ is called \emph{tempered} if there exists a tempered random variable $R$ such that
\[
  D(\omega)\subset B_{R(\omega)}(0) \faa \omega\in\Omega\,,
\]
where $D(\omega):=\{x\in  \R^d: (\omega, x)\in D\}$. $D$ is called compact if $D(\omega)\subset \R^d$ is compact for almost all $\omega\in\Omega$.
Denote by $\mathcal D$ the set of all compact tempered sets $D\in \mathcal F\otimes \mathcal B(\R^d)$. We now define the notion of a random attractor with respect to $\mathcal D$, see also \cite[Definition~14.3]{rk04}.
\begin{definition}[Random attractor]
   A set $A \in \mathcal D$ is called a \emph{random attractor} (with respect to $\mathcal D$) if the following two properties are satisfied.
	\begin{enumerate}
		\item[(i)] $A$ is $\varphi$-invariant, i.e.
		\begin{equation*}
		\varphi(t,\omega) A(\omega) = A (\theta_t  \omega) \fa t\ge 0\mbox{ and almost all } \omega \in \Omega\,.
		\end{equation*}
		\item[(ii)] For all $D \in \mathcal D$, we have
		\begin{equation*}
		\lim_{t \to \infty} \operatorname{dist} \big(\varphi(t, \theta_{-t} \omega)D(\theta_{-t}\omega), A(\omega)\big) = 0 \faa \omega\in\Omega\,,
		\end{equation*}
        where $\operatorname{dist}(E, F):= \sup_{x\in E}\inf_{y\in F} \|x-y\|$.
	\end{enumerate}
\end{definition}

Note that we require that random attractor is measurable with respect to $\mathcal F\otimes \mathcal B(\R^d)$, in contrast to a weaker statement normally used in the literature (see also \cite[Remark~4]{crauelkloeden15}).

The existence of random attractors is proved via so-called absorbing sets. A set $B\in\mathcal D$ is called an \emph{absorbing set} if for almost all $\omega\in\Omega$ and any $D \in \mathcal D$, there exists a $T>0$ such that
\[
\varphi(t,\theta_{-t}\omega)D(\theta_{-t}\omega)\subset B(\omega)\fa t\geq T\,.
\]
A proof of the following theorem can be found in \cite[Theorem~3.5]{flandolischmalfuss96}.

\begin{theorem}[Existence of random attractors]\label{ExistencePullback}
  Suppose that $(\theta,\varphi)$ is a continuous random dynamical system with an absorbing set $B$. Then there exists a unique random attractor $A$, given by
	\[
	A(\omega)
	:=
	\bigcap_{\tau\geq 0}\overline{\bigcup_{t\geq \tau} \varphi(t,\theta_{-t}\omega)B(\theta_{-t}\omega)}
	\faa \omega\in\Omega.
	\]
	Furthermore, $ \omega \mapsto A(\omega)$ is measurable with respect to $\mathcal{F}_{-\infty}^{0}$, i.e.~the past of the system.
\end{theorem}
\end{document}